\RequirePackage{fix-cm}
\documentclass[smallextended]{svjour3}       
\smartqed  
\usepackage{graphicx}
\usepackage{lipsum,amsfonts,graphicx,epstopdf,algorithmic}
\usepackage{amsopn,amsmath}
\usepackage{stmaryrd,multirow,tabularx,makecell,booktabs}
\usepackage{bm,tikz}

\newcommand{\tabincell}[2]{\begin{tabular}{@{}#1@{}}#2\end{tabular}}
\journalname{Optimization Letters}

\newcommand{\by}{\mathbf{y}}

\newcommand{\bu}{\mathbf{u}}
\newcommand{\bp}{\mathbf{p}}
\newcommand{\bU}{\mathbf{U}}

\newcommand{\bM}{\mathbf{M}}

\newcommand{\bR}{\mathbf{R}}

\newcommand{\bQ}{\mathbf{Q}}
\newcommand{\bx}{\mathbf{x}}

\newcommand{\ba}{\mathbf{a}}

\newcommand{\bb}{\mathbf{b}}

\newcommand{\bc}{\mathbf{c}}

\newcommand{\bw}{\mathbf{w}}

\newcommand{\bv}{\mathbf{v}}

\newcommand{\bGamma}{\mathbf{\Gamma}}

\newcommand{\bTheta}{\mathbf{\Theta}}
\newcommand{\bXi}{\mathbf{\Gamma}}
\newcommand{\bPi}{\mathrm{\Pi}}
\newcommand{\bPsi}{\mathrm{\Psi}}

\newcommand{\bOmega}{\mathbf{\Omega}}

\newcommand{\I}{\mathbf{I}}

\newcommand{\cI}{\mathcal{I}}

\newcommand{\cL}{\mathcal{L}}

\newcommand{\N}{\mathbb{N}}

\newcommand{\cA}{\mathcal{A}}

\newcommand{\R}{\mathbb{R}}

\newcommand{\dimsf}{\mathsf{dim}}

\newcommand{\dom}{\mathsf{dom}}
\newcommand{\zer}{\mathsf{zer}}

\newcommand{\prox}{\mathsf{prox}}
\newcommand{\bB}{\mathbf{B}}
\newcommand{\cT}{\mathcal{T}}

\newcommand{\cM}{\mathcal{M}}
\newcommand{\cK}{\mathcal{K}}
\newcommand{\cO}{\mathcal{O}}

\newcommand{\bt}{\mathbf{t}}
\newcommand{\bA}{\mathbf{A}}

\newcommand{\bbS}{\mathbb{S}}
\newcommand{\bs}{\mathbf{s}}

\newcommand{\bz}{\mathbf{z}}

\newcommand{\rank}{\mathsf{rank}}

\newcommand{\cQ}{\mathcal{Q}}
\newcommand{\cD}{\mathcal{D}}
\newcommand{\cS}{\mathcal{S}}

\newcommand{\cB}{\mathcal{B}}
\newcommand{\cG}{\mathcal{G}}
\newcommand{\cR}{\mathcal{R}}

\newcommand{\bxstar}{{\mathbf{x}}^\star}
\newcommand{\bcstar}{{\mathbf{c}}^\star}

\newcommand{\bustar}{{\mathbf{u}}^\star}
\newcommand{\bastar}{{\mathbf{a}}^\star}
\newcommand{\bbstar}{{\mathbf{b}}^\star}
\newcommand{\bzstar}{{\mathbf{z}}^\star}
\newcommand{\bpstar}{{\mathbf{p}}^\star}

\newcommand{\bchat}{\widehat{\mathbf{c}}}
\newcommand{\bphat}{\widehat{\mathbf{p}}}
\newcommand{\bvhat}{\widehat{\mathbf{v}}}
\newcommand{\buhat}{\widehat{\mathbf{u}}}

\newcommand{\bbhat}{\widehat{\mathbf{b}}}

\newcommand{\bxhat}{\widehat{\mathbf{x}}}

\newcommand{\bahat}{\widehat{\mathbf{a}}}

\newcommand{\bxtilde}{\tilde{\mathbf{x}}}
\newcommand{\bctilde}{\tilde{\mathbf{c}}}
\newcommand{\batilde}{\tilde{\mathbf{a}}}

\newcommand{\bptilde}{\tilde{\mathbf{p}}}

\newcommand{\bbtilde}{\tilde{\mathbf{b}}}

\newcommand{\cAtilde}{\tilde{\mathcal{A}}}

\newcommand{\be}{\begin{equation}}
\newcommand{\ee}{\end{equation}}

\begin{document}

\title{The operator splitting schemes revisited: primal-dual gap and degeneracy reduction   by a unified analysis}

\titlerunning{The primal-dual gap and degeneracy reduction of operator splitting schemes}        

\author{Feng Xue   
}


\institute{Feng Xue \at
              National key laboratory, Beijing,  China  \\
              \email{fxue@link.cuhk.edu.hk}           
}

\date{Received: date / Accepted: date}

\maketitle

\begin{abstract}
We revisit the operator splitting schemes proposed in a recent work of [{\it Some extensions of the operator splitting schemes based on Lagrangian and primal-dual: A unified proximal point analysis}, Feng Xue, Optimization, 2022,  doi:
10.1080/02331934.2022.2057309], and further analyze the convergence of the generalized Bregman distance and the primal-dual gap of these algorithms within a unified proximal point framework. The possibility of reduction to a simple resolvent is also discussed by exploiting the structure and possible degeneracy of the underlying metric. 
\end{abstract}

\section{Introduction}
\label{sec:intro}
The present paper discusses the operator splitting schemes for solving \cite[Eq.(1)]{dinh}
\be \label{primal}
\min_\bx  f(\bx) +  g(\bA\bx),
\ee
where $\bx \in \R^N$,  $\bA: \R^N \mapsto\R^{M}$ is a linear operator, $f: \R^{N} \mapsto \R \cup \{+\infty\} $ and $g: \R^{M} \mapsto \R \cup \{+\infty\} $ are proper, lower semi-continuous (l.s.c.),  convex (not necessarily smooth) and proximable\footnote{We say a convex function $f$ is {\it proximable}, if the proximity operator of $f$ has a closed-form representation or at least can be solved efficiently up to high precision \cite{condat_2013}. This property is also called `{\it simple}' \cite{condat_2013} or  `{\it with inexpensive proximity operator}' \cite{vanden_2014}.} functions.  As observed in \cite[Sect. 2.1]{drori} and \cite[Sect. 1.1]{fxue_gopt},  the problem \eqref{primal} also covers the minimization of the sum of multiple functions $g_i$ composed with linear operators $\bA_i$, i.e. $\min_\bx \sum_{i=1}^I g_i(\bA_i\bx)$, if we define  $\bA = \begin{bmatrix}
\bA_1  \\ \vdots \\  \bA_I \end{bmatrix}: \R^N \mapsto \R^M$, with  $\bA_i: \R^N \mapsto\R^{M_i}$ and  $M=\sum_{i=1}^I M_i$,  $g: \R^{M_1}  \times \cdots \times
\R^{M_I} \mapsto \R \cup \{+\infty\}: (\ba_1, \cdots, \ba_I) \mapsto \sum_{i=1}^I g_i(\ba_i)$.  This problem can be solved by many classes of operator splitting algorithms, e.g., Douglas-Rachford splitting (DRS) \cite{lions},  primal-dual splitting (PDS) \cite{pdhg,cp_2011}, the alternating direction method of multipliers (ADMM) \cite{glowinski_1,glowinski_2}, Bregman methods \cite{osher_tv,yin_2008,sb,zxq}, and so on\footnote{Note that the well-known {\it proximal forward-backward splitting} (PFBS) algorithm may not in general be applied to solve \eqref{primal}, since neither $f$ nor $g$  is  assumed to be differentiable with a Lipschitz continuous gradient.}.

In the recent work of \cite{fxue_gopt}, we gave a brief review of the typical ADMM-type and PDS algorithms, and presented a  unified proximal point  treatment. Following this work, we in this paper attempt to answer two important questions: 
\begin{enumerate}
\item {\it What value do these operator splitting algorithms attempt to minimize? Is it possible to analyze the convergence within the unified proximal point framework?} 

\item {\it Observing that some splitting strategies generate  auxiliary variables that maybe redundant in the iterations, can they be reduced to a simpler form with smallest number of variables? Is it possible to detect and reduce the degeneracy or redundancy under the unified proximal point analysis? }
\end{enumerate}

The contributions of this paper are in order.
\begin{itemize}
\item We show more possibilities of devising new algorithms than \cite{fxue_gopt} in a more systematic way, according to the metric structure. 

\item We give an affirmative answer to the above first question: what the operator splitting algorithms try to minimize is the {\it Bregman distance} of some convex functional (and the associated {\it primal-dual gap} under additional conditions). It can be inferred by the proximal point framework (Sect. \ref{sec_lag_gap}, \ref{sec_pds_gap} and \ref{sec_mix_gap}), which enables us to perform a unified gap analysis of all the  schemes developed in \cite{fxue_gopt}, that is much simpler than the existing case studies of specific algorithms, e.g. \cite[Theorem 2.1]{bot_jmiv_2014} and \cite[Theorem 1]{cp_2011}.

\item The unified proximal point interpretation paves a way for expressing many algorithms proposed in \cite{fxue_gopt} as a simple resolvent. More remarkably, by exploiting the metric degeneracy, some algorithms, particularly the standard ADMM/DRS, can be  reduced to a simple resolvent  involving only active variables and an explicit expression of the associated maximally monotone operator. This is also a case study of the degenerate analysis recently proposed in \cite{bredies_preprint}.
\end{itemize}

\section{Preliminaries}

\subsection{Notations and definitions} \label{sec_notations}
We use standard notations and concepts from convex analysis and variational analysis, which, unless otherwise specified, can all be found in the classical and recent monographs \cite{plc_book,beck_book,rtr_book,rtr_book_2}.

A few more words about our  notations are as follows.  The class of symmetric and positive semi-definite/definite matrices is denoted by $\bbS_+$ or  $\bbS_{++}$, respectively.  We use the boldface uppercase to denote a matrix, e.g., $\bM$, the calligraphic uppercase to denote a block-structured matrix or operator, e.g., $\cM$.
The identity operator and identity matrix of size $N \times N$ are denoted  by $\cI$ and  $\I_N$.  The $\bM$-norm with $\bM \in \bbS_+$ is defined as: $\|\bx\|_\bM^2 := \langle \bx | \bM \bx \rangle$. 

The generalized proximity operator, denoted by $\prox_{f}^\bM$ is defined as  $\prox_{f}^\bM:  \bx \mapsto  \arg \min_\bu $  $ f(\bu)+ \frac{1}{2} \|\bu-\bx\|_\bM^2$, with  $\bM \in \bbS_+$ as an induced metric \cite[Eq.(4)]{vu_2015}, \cite[Definition 2.3]{pesquet_2016}.  If $\bM = \frac{1}{ \tau} \I$ (i.e. the scalar case), the generalized proximity operator reduces to ordinary one, denoted by $\prox_{\tau f}:  \bx \mapsto  \arg \min_\bu $  $ f(\bu)+ \frac{1}{2\tau} \|\bu-\bx\|^2$  \cite[Definition 12.23]{plc_book}, \cite[Eq.(2.13)]{plc}. 

The classical Bregman distance associated with the function $\varphi$ between $\bx$ and $\by$ is defined as $D_\varphi(\bx,\by)=\varphi(\bx)-\varphi(\by) - \langle \nabla \varphi(\by) | \bx - \by\rangle$, which requires the function $\varphi$ to be differentiable and strictly convex. It was then extnded to the context of proper, l.s.c. and convex function $\varphi$ along a direction of $\bv$ within its subdifferential $\partial \varphi$ \cite{kiwiel}:
\be \label{breg_def}
D_\varphi^\bv(\bx,\by)=\varphi(\bx)-\varphi(\by) - \langle \bv| \bx - \by\rangle,\quad \bv \in \partial \varphi(\by).
\ee
This plays a central role in various Bregman algorithms, e.g.,  \cite{cjf_1,sb,osher_tv,zxq}. \cite{burachik_svaa} further proposed two types of the generalized Bregman distance \eqref{breg_def}:
\[
\left\{ \begin{array}{lll}
D_\varphi^\sharp (\bx,\by) &=& \varphi(\bx)-\varphi(\by) + \sup_{\bv \in \partial \by} \langle \bv | \by - \bx\rangle,  \\
D_\varphi^\flat (\bx,\by) &=& \varphi(\bx)-\varphi(\by) + \inf_{\bv \in \partial \by} \langle \bv| \by - \bx\rangle ,
\end{array} \right.
\]
which are the upper and lower bounds of the Bregman distance generated by $\varphi$.

\subsection{Some existing results of proximal point algorithm}
\label{sec_gppa}
The generalized proximal point algorithm (PPA) is given as
 \be \label{gppa}
\left\lfloor \begin{array}{llll}
\bf 0  & :\in  & \cA \tilde{\bc}^{k} + \cQ 
(\tilde{\bc}^{k}- \bc^k),  & \text{ \rm (proximal step)} \\ 
\bc^{k+1} & := & \bc^k + \cM ( \tilde{\bc}^{k} - \bc^k), 
 & \text{ \rm (relaxation step)}
 \end{array}  \right.
\ee 
where $\cA$ is a (possibly set-valued) monotone operator, $\cQ$ is a  metric, $\cM$ is an {\it invertible} relaxation matrix. The convergence of \eqref{gppa} has been extensively studied in the contexts of DRS, PDHG and multi-block ADMM algorithms, e.g., \cite{hbs_yxm_2018,hbs_jmiv_2017,hbs_yxm_2015}, and recently  revisited in our recent works of \cite{fxue_2,fxue_gopt}. 

Here, we restate the main results therein with more straightforward proofs. More than that, our analysis admits any choices of $\cA$, $\cQ$ and $\cM$,  not limited to any specific algorithms. 

\begin{lemma} \label{l_gppa}
Let  $\{\bc^k\}_{k\in\N}$ be a sequence generated by \eqref{gppa} and $\bc^\star \in \zer \cA$. Denote $\cS: = \cQ \cM^{-1}$, $\cG := \cQ + \cQ^{\top} - \cM^\top \cQ$. Denote the operator $\cT: = (\cA+\cQ)^{-1} \cQ$, and $\cR:=\cI - \cT$.  If  $\cA$ is maximally monotone and  $\cS, \cG \in \bbS_{++}$,  then, the following hold.

{\rm (i)}  $  \big\| \bc^{k+1} - \bcstar \big\|_\cS^2
\le   \big\| \bc^k - \bcstar \big\|_\cS^2
-  \big \| \bc^k -  \bc^{k+1}  \big \|
_{\cM^{-\top} \cG \cM^{-1} }^2 $;

{\rm (ii)}  $ \big \langle  \cM^\top \cQ \cR \bc^k \big|  \cR \bc^k - \cR \bc^{k+1}  \big \rangle   \ge  \frac{1}{2}
\big\|   \cR \bc^k - \cR  \bc^{k+1}  \big\|
_{\cQ+\cQ^\top}^2$;

{\rm (iii)}  $  \big \| \bc^k - \bc^{k+1} \big\|_\cS^2 - 
\big\| \bc^{k+1} - \bc^{k+2} \big\|_\cS^2   
 \ge   \big\|   \cR \bc^k -  \cR \bc^{k+1}  \big\|_\cG^2  $.
\end{lemma}

\begin{proof}
(i) Based on \eqref{gppa}, we have:
\begin{eqnarray} 
0 & \le &  \big \langle \cA  \bctilde^{k} - \cA \bcstar \big| 
\bctilde^k - \bcstar \big \rangle  \quad
\text{by monotonicity of $\cA$}
\nonumber \\
&=&  \big \langle \cQ (\bc^k -  \bctilde^{k}) \big| 
\bctilde^k - \bcstar \big \rangle
\quad  \text{ by \eqref{gppa} and ${\bf 0} \in \cA \bbstar $ }
\nonumber \\
&=&  \big \langle \cQ \cM^{-1} (\bc^k -  \bc^{k+1}) \big| 
\bc^k + \cM^{-1} (\bc^{k+1} - \bc^k) - \bcstar \big \rangle
\quad  \text{ by \eqref{gppa} }
\nonumber \\
&=&  \big \langle \cS( \bc^k -  \bc^{k+1} ) \big|  \bc^k - \bcstar \big \rangle - \frac{1}{2} \big \| \bc^k -  \bc^{k+1}  \big \|_{\cM^{-\top} \cS + \cS \cM^{-1} }^2 
\nonumber \\
&=& \frac{1}{2} \big\| \bc^k - \bcstar \big\|_\cS^2
- \frac{1}{2} \big\| \bc^{k+1} - \bcstar \big\|_\cS^2
- \frac{1}{2} \big \| \bc^k -  \bc^{k+1}  \big \|
_{\cM^{-\top} \cG \cM^{-1} }^2, 
\quad \text{by $\cS\in \bbS_{++}$.}
\nonumber 
\end{eqnarray}

\vskip.2cm
(ii) By \cite[Lemma 2.6]{fxue_1}, we obtain:
\[
 \big \langle     \bc^k - \bc^{k+1} \big|  
 \cQ \cR \bc^k - \cQ \cR \bc^{k+1}  \big \rangle  \ge 
\big\|   \cR \bc^k - \cR  \bc^{k+1}  \big\|_\cQ^2
= \frac{1}{2} \big\|   \cR \bc^k - \cR  \bc^{k+1}  \big\|_{\cQ + \cQ^\top}^2.
\]
Then, (ii) follows from  $\bc^k - \bc^{k+1} = \cM\cR \bc^k$.

\vskip.2cm
(iii) We develop:
\begin{eqnarray}  \label{dd}
& &  \big \| \bc^k - \bc^{k+1} \big\|_\cS^2 - 
\big\| \bc^{k+1} - \bc^{k+2} \big\|_\cS^2   
\nonumber \\
& = &  \big \|\cM \cR \bc^k \big\|_\cS^2 - 
\big\| \cM \cR\bc^{k+1} \big\|_\cS^2  
\quad \text{by \eqref{gppa} } 
\nonumber \\
&=& 2 \big \langle  \cM^\top \cS \cM  \cR \bb^k \big| 
 \cR \bc^k - \cR \bc^{k+1}   \big \rangle   
- \big\| \cR\bc^k - \cR\bc^{k+1} \big\|^2
_{\cM^\top   \cS \cM } 
\nonumber \\ 
& \ge &  \big\|   \cR \bc^k -  \cR \bc^{k+1}  \big\|_\cG^2,  \qquad 
\text{by Lemma \ref{l_gppa}--(ii) and $\cQ  =\cS \cM$.}   
\nonumber 
\end{eqnarray}
\end{proof}
\begin{remark}
Lemma \ref{l_gppa}--(i) can be found in \cite[Theorem 1]{hbs_yxm_2018}, \cite[Theorem 4.1]{hbs_yxm_2015} and \cite[Theorem 3.2]{hbs_jmiv_2017}. The item (ii) is same as \cite[Lemma 3]{hbs_yxm_2018}, \cite[Lemma 5.3]{hbs_yxm_2015} and \cite[Lemma 5.4]{hbs_jmiv_2017}. 
The item (iii) is a restatement of \cite[Theorem 5]{hbs_yxm_2018} and \cite[Theorem 5.1]{hbs_yxm_2015}. The proof presented here outlines the key ingredients only. Refer to \cite[Lemma 5.2]{fxue_2} for more details.
\end{remark}

Then, the convergence properties of \eqref{gppa} are  given as follows. 
\begin{theorem} [Convergence in terms of metric distance] \label{t_gppa}
Under the notations and assumptions of Lemma \ref{l_gppa}, then the following hold.

{\rm (i) [Basic convergence]} There exists $\bcstar\in \zer \cA$, such that $\bc^k \rightarrow \bcstar$, as $k \rightarrow \infty$. 

{\rm (ii) [Asymptotic regularity]}  $\|  \bc^{k } - \bc^{k+1 } \|_\cS$ has the pointwise  convergence rate of $\cO(1/\sqrt{k})$, i.e.,
\[
\big\| \bc^{k+1 } - \bc^{k }  \big\|_\cS 
\le \frac{1}{ \sqrt{k+1 } } \sqrt{\frac
{\lambda_{\max}(\cS )} 
{\lambda_{\min}(\cM^{-\top} \cG \cM^{-1})}  }  
\big\|\bc^{0} -\bcstar \big\|_\cS,
\quad \forall k \in \N,
\]
 where $\lambda_{\max}$ and $\lambda_{\max}$ denote the largest and smallest eigenvalues of a matrix. 
\end{theorem}

\begin{proof} 
(i)  The basic convergence is established based on Lemma \ref{l_gppa}--(i), invoking Opial's lemma \cite[Lemma 2.47]{plc_book}.

\vskip.2cm 
(ii) In view  of Lemma \ref{l_gppa}--(i), we have:
\be \label{xx1}
\big\| \bc^{i+1} -\bc^\star \big\|_\cS^2 \le 
\big\| \bc^{i} - \bc^\star \big\|_\cS^2 - \frac{\lambda_{\min}( \cM^{-\top} \cG \cM^{-1})}
{\lambda_{\max}(\cS )} \big\| \bc^{i } - \bc^{i+1} \big\|_\cS^2.
\ee

Finally, (ii) is obtained, by summing up \eqref{xx1} from $i=0$ to $k$ and noting that the sequence  $\{\| \bc^{i} - \bc^{i+1} \|_\cS\}_{i\in\N} $ is non-increasing (by Lemma \ref{l_gppa}--(iii)).
\end{proof}

\begin{remark}
Refer to \cite[Theorem 5.3]{fxue_2} for more details.
The non-ergodic rate of asymptotic regularity (ii) has also been established in \cite[Theorem 6]{hbs_yxm_2018}, \cite[Theorem 6.1]{hbs_yxm_2015} and \cite[Theorem 5.5]{hbs_jmiv_2017}. 
\end{remark}

In particular, if $\cM=\cI$, the scheme \eqref{gppa} reduces to a standard PPA: ${\bf 0} \in \cA \bc^{k+1} +\cQ(\bc^{k+1} - \bc^k)$, which can be rewritten as
\be \label{ppa}
\bc^{k+1} := (\cA+\cQ)^{-1} \cQ \bc^k ,
\ee
whose convergence is given below.

\begin{corollary} [Convergence of standard PPA] \label{c_gppa}
Given the scheme \eqref{ppa} with maximally monotone $\cA$ and metric $\cQ \in \bbS_{++}$, the following hold.

{\rm (i) [Basic convergence]} There exists $\bcstar\in \zer \cA$, such that $\bc^k \rightarrow \bcstar$, as $k \rightarrow \infty$. 

{\rm (ii) [Asymptotic regularity]}  $\|  \bc^{k } - \bc^{k+1 } \|_\cQ$ has a pointwise  convergence rate of $\cO(1/\sqrt{k})$, i.e.
\[
\big\| \bc^{k+1 } - \bc^{k }  \big\|_\cQ 
\le \frac{1}{ \sqrt{k+1 } }   
\big\|\bc^{0} -\bcstar \big\|_\cQ,
\quad \forall k \in \N,
\]

{\rm (iii) [Resolvent]} The scheme \eqref{ppa} can be rewritten as a resolvent form:
\[
\bc^{k+1} :=  (\cI + \cQ^{-1} \circ \cA)^{-1} \bc^k = 
\cQ^{-\frac{1}{2} } \big(\cI +  \cQ^{-\frac{1}{2} } \circ \cA \circ \cQ^{-\frac{1}{2} } \big)^{-1}   \cQ^{\frac{1}{2} }   \bc^k.
\]
\end{corollary}
\begin{proof}
(i) and (ii) follow from Theorem \ref{t_gppa}.

(iii) \cite[Lemma 2.1-(iii)]{fxue_2}.
\end{proof}
 
All the results presented in Sect. \ref{sec_gppa} require the associated metrics $\cQ$ (or  $\cS,\cG$) to be {\it strictly} positive definite. However, based on a recent analysis of \cite{bredies_preprint}, this condition can be sometimes loosened to positive {\it semi-}definite in the applications to operator splitting algorithms, which leads to some interesting reductions by removing redundant variables. See Sect. \ref{sec_red_lag} and  \ref{sec_red_pds} for detailed discussions.

\subsection{An extension of  Moreau's decomposition identity}
The following result extends the classical Moreau's decomposition identity (see, for instance, \cite[Eq.(2.21)]{plc}) to arbitrary linear operator $\bA$, and  links  the proximity operator of the infimal postcomposition of $f$ by $\bA$ to that of the conjugate $f^*$. The the notion of `infimal postcomposition' was recently studied in \cite{arias_parallel} in details.

\begin{lemma} \label{l_duality}
Given a proper, l.s.c. and convex function $f: \R^N \mapsto
 \R\cup \{+\infty\}$ and arbitrary matrix $\bA: \R^N \mapsto \R^M$, the following holds: 
\[
\prox_{\bA \triangleright f} + \prox_{f^* \circ \bA^\top}
=\I_M,
\]
where $\bA \triangleright f$ denotes the infimal postcomposition of $f$ by $\bA$, defined as $\bA \triangleright f: \R^M \mapsto \R: \bt \mapsto  \min_{\bA \bx=\bt } f(\bx) $. 
\end{lemma}
\begin{proof}
First, incorporating a hard constraint of $\bt=\bA  \bx$  \cite[Sect. 3]{self_eq}, we have:
\begin{eqnarray} \label{e4}
\min_\bx  f(\bx) + \frac{1}{2} \|\bA\bx-\bu\|^2 
&=&  \min_{\bx,\bt}  f(\bx) + 
\frac{1}{2} \|\bt-\bu\|^2 + \iota_{ \{\bx: \bA \bx = \bt \} }(\bt)
\nonumber \\
&=&  \min_{\bt}  F(\bt) + 
\frac{1}{2} \|\bt-\bu\|^2 ,
\end{eqnarray}
where $\iota_C$ is an indicator function of a set $C$, the function $F(\bt) := \min_\bx f(\bx) +\iota_{ \{\bx: \bA \bx=\bt\} }(\bx) 
= \min_{\bA \bx=\bt } f(\bx)  $ is the so-called infimal postcomposition of $f$ by $\bA$, simply denoted as $F =\bA \triangleright f$ \cite[Definition 12.34]{plc_book}. \eqref{e4} implies that $\bt^\star = \prox_{F}(\bu) = \bA \bx^\star$, where $\bx^\star = \arg \min_\bx  f(\bx) + \frac{1}{2} \|\bA\bx-\bu\|^2 $.

Then,  by Fenchel duality, the above   is equivalent to a saddle-point problem:
\[
\min_\bt \max_\bs  \langle \bs | \bt \rangle  - F^*(\bs) + \frac{1}{2}  \|\bt  - \bu\|^2 .
\]
Exchanging the order of min and max, we have:
\[
 \max_\bs \min_\bt   - F^*(\bs) + \frac{1}{2}  \|\bt  - \bu +\bs\|^2 - \frac{1}{2}\|\bs\|^2 +\langle \bu | \bs\rangle,
 \]
which yields that $\bt^\star    =\bu - \bs^\star$, where
$\bs^\star  =  \arg \min_\bs  F^*(\bs) +\frac{1}{2}  \|\bs - \bu \|^2 = \prox_{F^*} (\bu)$. Finally, the proof is completed by noting that $F^* = f^* \circ \bA^\top$ \cite{self_eq}.
\end{proof}

\section{Operator splitting based on Lagrangian}
\subsection{The Lagrangian schemes and their PPA interpretations}
\label{sec_lag_algo}
First, we consider the Lagrangian of \eqref{primal} \cite[Eq.(13)]{fxue_gopt}:
\be \label{lag}
\cL (\bx,\ba,\bp)  := 
f(\bx) + g(\ba)  + \bp^\top (\bA \bx - \ba),
\ee 
or generalized augmented Lagrangian:
\be \label{aug_lag}
\cL_\bXi (\bx,\ba,\bp)  := 
f(\bx) + g(\ba)  + \bp^\top (\bA \bx - \ba)
+\frac{1}{2} \big\| \bA \bx - \ba\big\|_{\bXi}^2,
\ee 
which extends the standard augmented Lagrangian \cite[Eq.(3)]{fxue_gopt} from the scalar penalty parameter $\gamma$ to the matrix metric $\bXi$.

Then, similar to \cite{pesquet_2016,frankel_2015,vu_2015,alves_2018},  defining the proximal metrics by  $\bM: \R^N \mapsto \R^N$,  $\bOmega: \R^M \mapsto \R^M$,  $\bXi: \R^M \mapsto \R^M$, the alternating optimization of \eqref{lag} or \eqref{aug_lag} yields the algorithms listed in Table 1. \texttt{LAG-I,II,V,VI} and \texttt{VII} can be found in \cite[Sect. 3 and 4]{fxue_gopt}, and are extended to general proximal metrics here.  Table 2 shows the PPA reinterpretations of the  schemes. One can check the PPA fitting by the similar procedure with  \cite{bredies_2017,mafeng_2018,bai_2018,hbs_2018},  verify the convergence condition for each algorithm (shown in Table 3) by computing the corresponding $\cS$ and $\cG$ by Theorem \ref{t_gppa}, and further write down the specific convergence property of asymptotic regularity, which are omitted here.  Also note that:
\begin{itemize}
\item \texttt{LAG-I}  and  \texttt{LAG-II}  correspond to symmetric $\cQ$ (without relaxation); 

\item   \texttt{LAG-III} and  \texttt{LAG-IV} correspond to upper triangular $\cQ$; 

\item  \texttt{LAG-V} and  \texttt{LAG-VI} correspond to lower triangular $\cQ$; 

\item  \texttt{LAG-VII}  corresponds to skew-symmetric  $\cQ$. 
\end{itemize}

\begin{table} [h!] \label{table_lag_algo}
\centering
\caption{The proposed Lagrangian-based algorithms }
\vspace{-.5em}
\hspace*{-.2cm}
\resizebox{.97\columnwidth}{!} {
\begin{tabular}{|l|l|} 
    \Xhline{1.2pt}  
    name  & iterative scheme \\
\hline 
 \tabincell{l}{ \texttt{LAG-I} \\ \cite[Eq.(14)]{fxue_gopt} } 
  & $ \left\lfloor  \begin{array}{lll}
\bx^{k+1} &: = & \prox_{f}^\bM \big( \bx^k -\bM^{-1} \bA^\top \bp^k \big) \\
  \ba^{k+1} &: =  & \prox_{g}^{\bOmega } (\ba^k + \bOmega^{-1} \bp^k  )    \\ 
\bp^{k+1}  & := & \bp^k + \bXi  
 \big(  \bA (2 \bx^{k+1}  -\bx^k ) - (2\ba^ {k+1} -\ba^k)  \big)    \end{array} \right. $ \\  
 \hline
 \tabincell{l}{ \texttt{LAG-II} \\ \cite[Eq.(24)]{fxue_gopt} } 
  & $ \left\lfloor  \begin{array}{lll}
\bx^{k+1} & := &  \prox_{f}^{ \bM+ \bA^\top \bXi \bA }\big( ( \bM + \bA^\top\bXi \bA)^{-1} ( 
\bM \bx^k  +\bA^\top  \bXi \ba^{k} - \bA^\top \bp^k  ) \big)  \\
\bp^{k+1}  & = & \bp^k + \bXi  \big(  \bA  \bx^{k+1} - \ba^ {k}  \big)  \\ 
  \ba^{k+1} & =  & \prox_{g}^{\bOmega } \big(   \ba^k +   \bOmega^{-1}  (2 \bp^{k+1} - \bp^k) \big)  
 \end{array} \right. $ \\
 \hline
 \texttt{LAG-III} & $ \left\lfloor  \begin{array}{lll}
  \ba^{k+1}   &: =  & \prox_{g}^{\bOmega} ( \ba^k + \bOmega^{-1}  \bp^k  )    \\ 
\bx^{k+1} &: = & \prox_f^{\bM+\bA^\top \bXi \bA} \big( (\bM+\bA^\top \bXi \bA)^{-1}
(\bM \bx^k - \bA^\top \bp^k +\bA^\top \bXi \ba^{k+1})  \big)  \\ 
\bp^{k+1}  & := & \bp^k + \bXi \big(  \bA \bx^{k+1} + \ba^{k}  -2\ba^{k+1} \big) 
 \end{array} \right. $\\
 \hline
 \texttt{LAG-IV}  & $ \left\lfloor  \begin{array}{lll}
\bx^{k+1} &: = & \prox_f^{\bM} \big(   \bx^k - \bM^{-1} \bA^\top   \bp^{k} \big)  \\ 
  \ba^{k+1}   &: =  & \prox_{g}^{\bOmega+\bXi} \big(
  (\bOmega+\bXi )^{-1} (\bOmega \ba^k + \bXi \bA \bx^{k+1} +\bp^k ) \big)    \\ 
\bp^{k+1}  & := & \bp^k + \bXi \big(  \bA (2\bx^{k+1}
-\bx^k)    - \ba^{k+1} \big) 
 \end{array} \right. $ \\
 \hline
  \tabincell{l}{ \texttt{LAG-V} \\ \cite[Eq.(19)]{fxue_gopt} } 
 & $  \left\lfloor  \begin{array}{lll}
\bx^{k+1} & := &  \prox_{f}^{ \bM+ \bA^\top \bXi \bA }\big( ( \bM + \bA^\top\bXi \bA)^{-1} ( 
\bM \bx^k  +\bA^\top  \bXi  \ba^{k} - \bA^\top \bp^k  ) \big)  \\
  \ba^{k+1}   &: =  & \prox_{g}^{\bOmega+\bXi } \big(
  (\bOmega+\bXi )^{-1} (\bOmega \ba^k + \bXi  \bA \bx^{k} +\bp^k ) \big)    \\ 
\bp^{k+1}  & = & \bp^k + \bXi \big(  \bA \bx^{k+1} - \ba^ {k+1}  \big)   
 \end{array} \right. $  \\
 \hline
  \tabincell{l}{ \texttt{LAG-VI} \\ \cite[Eq.(20)]{fxue_gopt} } 
  & $  \left\lfloor  \begin{array}{lll}
\bx^{k+1} & := &  \prox_{f}^{ \bM+ \bA^\top \bXi \bA }\big( ( \bM + \bA^\top\bXi \bA)^{-1} ( 
\bM \bx^k  +\bA^\top  \bXi \ba^{k} - \bA^\top \bp^k  ) \big)  \\
  \ba^{k+1}   &: =  & \prox_{g}^{\bOmega+\bXi } \big(
  (\bOmega+\bXi )^{-1} (\bOmega \ba^k + \bXi  \bA \bx^{k+1} +\bp^k ) \big)    \\ 
  \bp^{k+1}  & := & \bp^k + \bXi  \big(  \bA \bx^{k+1} -
 \ba^ {k+1}  \big)  
 \end{array} \right.  $  \\
 \hline
  \tabincell{l}{ \texttt{LAG-VII} \\ \cite[Eq.(18)]{fxue_gopt} } 
 & $  \left\lfloor  \begin{array}{lll}
\bxtilde^{k } &: = & \prox_f^{\bM} \big(   \bx^k - \bM^{-1} \bA^\top   \bp^{k} \big)  \\ 
  \batilde^{k}   &: =  & \prox_{g}^{\bOmega} \big(
 \ba^k + \bOmega^{-1}  \bp^k  \big)    \\ 
\bx^{k+1} &:= &  \bxtilde^k - \bM^{-1}\bA^\top \bXi (\bA\bx^k - \ba^k) \\
\ba^{k+1} &:= &  \batilde^k + \bOmega^{-1}  \bXi (\bA\bx^k - \ba^k) \\ 
\bp^{k+1}  & := & \bp^k + \bXi \big(  \bA \bxtilde^{k }
   - \batilde^{k} \big)  
 \end{array} \right. $  \\
    \Xhline{1.2pt}  
     \end{tabular}  } 
\vskip 0.5em
\end{table}

\begin{table} [h!] \label{table_lag}
\centering
\caption{The PPA reinterpretations of the Lagrangian-based schemes }
\vspace{-.5em}
\hspace*{-.2cm}
\resizebox{.97\columnwidth}{!} {
\begin{tabular}{||c||c|c|c|c||} 
    \Xhline{1.2pt}  
schemes & $\bc$ & $\cA$ & $\cQ$ & $\cM$  \\
\hline
 \texttt{LAG-I}  &
\multirow{3}{*} {  \tabincell{c}{ \\ \\  \\ \\ \\ \\ 
\\ \\ \\  $\begin{bmatrix}
\bx \\  \ba \\ \bp \end{bmatrix} $ } } & 
\multirow{3}{*}{ \tabincell{c}{ \\ \\  \\ \\ \\ \\ 
\\ \\ \\   $\begin{bmatrix}
\partial f & \bf 0  & \bA^\top \\
\bf 0 & \partial g  & -\I_{M}   \\
-\bA  & \I_{M} & \bf 0   
\end{bmatrix} $} } &  
$ \begin{bmatrix}
\bM   & \bf 0 &  -\bA^\top   \\
\bf 0 & \bOmega  &  \I_{M}   \\
-\bA &  \I_{M}  &  \bXi^{-1}     
\end{bmatrix} $   &  
\multirow{2}{*}{ \tabincell{c}{ \\ \\ $\I_{2M+N} $ } }  \\ 
 \cline{1-1}   \cline{4-4}   
 \texttt{LAG-II}  & &  &  
$ \begin{bmatrix}
\bM   & \bf 0 &   \bf 0   \\
\bf 0 & \bOmega   &  -\I_M  \\
\bf 0 &  -\I_M  &  \bXi^{-1}     
\end{bmatrix}   $ &    \\ 
 \cline{1-1}   \cline{4-5} 
 \texttt{LAG-III}  & &  &  
$ \begin{bmatrix}
\bM   & \bf 0 &  \bf 0   \\
\bf 0 & \bOmega  &  \I_{M}   \\
 \bf 0 &  \bf 0  &  \bXi^{-1} 
\end{bmatrix} $   &  
$ \begin{bmatrix}
\I_N   & \bf 0 &  \bf 0    \\
\bf 0 & \I_M  &  \bf 0   \\
\bf 0  & -\bXi  &  \I_M     
\end{bmatrix} $    \\ 
 \cline{1-1}   \cline{4-5} 
 \texttt{LAG-IV}  & &  &  
$ \begin{bmatrix}
\bM   & \bf 0 &  -\bA^\top   \\
\bf 0 & \bOmega  &  \bf 0  \\
 \bf 0 &  \bf 0  &  \bXi^{-1} 
\end{bmatrix} $   &  
$ \begin{bmatrix}
\I_N   & \bf 0 &  \bf 0    \\
\bf 0 & \I_M  &  \bf 0   \\
\bXi \bA  &  \bf 0  &  \I_M     
\end{bmatrix} $   \\ 
 \cline{1-1}   \cline{4-5} 
 \texttt{LAG-V}  & &  &  
$ \begin{bmatrix}
\bM+\bA^\top \bXi \bA   & \bf 0 &   \bf 0   \\
\bf 0 & \bOmega +\bXi   &  \bf 0  \\
\bA &  -\I_M  &  \bXi^{-1}     
\end{bmatrix} $ & 
$ \begin{bmatrix}
\I_N   & \bf 0 & \bf 0  \\
\bf 0 & \I_M  &  \bf 0 \\
\bXi \bA &  -\bXi  &  \I_M      
\end{bmatrix}$   \\ 
 \cline{1-1}   \cline{4-5} 
 \texttt{LAG-VI}  & &  &  
$  \begin{bmatrix}
\bM   & \bf 0 &   \bf 0   \\
\bf 0 & \bOmega +\bXi  &  \bf 0  \\
\bf 0 &  -\I_M  &  \bXi^{-1}     
\end{bmatrix} $  & $ \begin{bmatrix}
\I_N   & \bf 0 & \bf 0  \\
\bf 0 & \I_M  &  \bf 0 \\
\bf 0 &  -\bXi  &  \I_M      
\end{bmatrix}$     \\  
 \cline{1-1}   \cline{4-5} 
 \texttt{LAG-VII} & &  &  
$  \begin{bmatrix}
\bM   & \bf 0 &  -\bA^\top   \\
\bf 0 & \bOmega  &  \I_M \\
\bA &  -\I_M  &  \bXi^{-1}     
\end{bmatrix} $ & 
$ \begin{bmatrix}
\I_N   & \bf 0 & -\bM^{-1} \bA^\top \\
\bf 0 & \I_M  &  \bOmega^{-1} \\
\bXi \bA &  -\bXi  &  \I_M      
\end{bmatrix}$     \\ 
    \Xhline{1.2pt}  
     \end{tabular}  } 
\vskip 0.5em
\end{table}

\begin{table} [h!] \label{table_lag_sq}
\centering
\caption{The corresponding $\cS$ and $\cG$ of the Lagrangian-based schemes }
\vspace{-.5em}
\hspace*{-.2cm}
\resizebox{.97\columnwidth}{!} {
\begin{tabular}{||c||c|c|c||} 
    \Xhline{1.2pt}  
schemes & $\cS$ & $\cG$ &  convergence condition \\
\hline
 \texttt{LAG-I}  &  $ \begin{bmatrix}
\bM   & \bf 0 &  -\bA^\top   \\
\bf 0 & \bOmega  &  \I_{M}   \\
-\bA &  \I_{M}  &  \bXi^{-1}     
\end{bmatrix} $ &  $ \begin{bmatrix}
\bM   & \bf 0 &  -\bA^\top   \\
\bf 0 & \bOmega  &  \I_{M}   \\
-\bA &  \I_{M}  &  \bXi^{-1}     
\end{bmatrix} $ &  \tabincell{c}{  $\bM,\bOmega,\bXi \in \bbS_{++}$ \\ 
 $\bXi^{-1}  \succ \bA \bM^{-1}
\bA^\top + \bOmega^{-1}$ }  \\ 
\hline
 \texttt{LAG-II}  & $ \begin{bmatrix}
\bM   & \bf 0 &   \bf 0   \\
\bf 0 & \bOmega   &  -\I_M  \\
\bf 0 &  -\I_M  &  \bXi^{-1}     
\end{bmatrix}   $  & $ \begin{bmatrix}
\bM   & \bf 0 &   \bf 0   \\
\bf 0 & \bOmega   &  -\I_M  \\
\bf 0 &  -\I_M  &  \bXi^{-1}     
\end{bmatrix}   $  &  \multirow{2}{*}{  \tabincell{c}{ \\   $\bM \in \bbS_{+}$ \\ $\bOmega,\bXi \in \bbS_{++}$ \\ 
 $\bOmega   \succ  \bXi $ } }   \\ 
\cline{1-3}
 \texttt{LAG-III}  & $  \begin{bmatrix}
\bM   &  \bf  0  &  \bf 0   \\
 \bf 0   &  \bOmega + \bXi &  \I_M \\
  \bf 0  &   \I_M  &  \bXi^{-1}  \end{bmatrix} $ & 
$  \begin{bmatrix}
\bM   &  \bf  0  &  \bf 0   \\
 \bf 0   &  \bOmega &  \I_{M}  \\
  \bf 0  & \I_{M}  & \bXi^{-1}   \end{bmatrix}$ &   \\ 
  \hline
  \texttt{LAG-IV}  &  $ \begin{bmatrix}
\bM+\bA^\top \bXi \bA   & \bf 0 &  -\bA^\top   \\
\bf 0 & \bOmega  &  \bf 0 \\
-\bA &  \bf 0  &  \bXi^{-1}   
\end{bmatrix} $ &  $  \begin{bmatrix}
\bM   & \bf 0 &  -\bA^\top \\
\bf 0 & \bOmega  &  \bf 0 \\
-\bA &  \bf 0  &  \bXi^{-1}     
\end{bmatrix} $ & \tabincell{c}{ $\bM, \bXi \in \bbS_{++}$ \\ $\bOmega \in \bbS_{+}$ \\ 
 $\bM \succ \bA^\top \bXi \bA $ } \\ 
  \hline
  \texttt{LAG-V}  &  $   \begin{bmatrix}
\bM +\bA^\top \bXi \bA   & \bf 0 & \bf 0 \\
\bf 0 & \bOmega+\bXi   &  \bf 0 \\
\bf 0 &  \bf 0  &  \bXi^{-1}   
\end{bmatrix} $ &  $  \begin{bmatrix}
\bM   &  \bA^\top\bXi  &  \bf 0 \\
 \bXi  \bA & \bOmega   &  \bf 0  \\
\bf 0 &   \bf 0  &  \bXi^{-1}    
\end{bmatrix} $ &  \tabincell{c}{ $\bM,\bOmega,\bXi \in \bbS_{++}$ \\  $\bM  \succ \bA^\top \bXi   \bOmega^{-1} \bXi  \bA $ }  \\ 
  \hline
  \texttt{LAG-VI}  &  $   \begin{bmatrix}
\bM    & \bf 0 & \bf 0 \\
\bf 0 & \bOmega+\bXi  &  \bf 0 \\
\bf 0 &  \bf 0  &  \bXi^{-1}     
\end{bmatrix} $ &  $   \begin{bmatrix}
\bM    & \bf 0 & \bf 0 \\
\bf 0 & \bOmega   &  \bf 0 \\
\bf 0 &  \bf 0  &  \bXi^{-1}     
\end{bmatrix}  $ &   \tabincell{c}{ $\bM,\bOmega \in \bbS_{+}$ \\  $\bXi \in \bbS_{++} $ }\\ 
  \hline
  \texttt{LAG-VII}  &  $   \begin{bmatrix}
\bM    & \bf 0 & \bf 0 \\
\bf 0 & \bOmega  &  \bf 0 \\
\bf 0 &  \bf 0  &  \bXi^{-1}     
\end{bmatrix} $ &  $  \begin{bmatrix}
\bM -\bA^\top \bXi\bA  & \bA^\top \bXi  & \bf 0  \\
 \bXi \bA  & \bOmega-\bXi   &  \bf 0  \\
\bf 0 &   \bf  0  & \bXi^{-1} - \bA\bM^{-1}\bA^\top -\bOmega^{-1}
\end{bmatrix} $ &   \tabincell{c}{ $\bM,\bOmega,\bXi \in \bbS_{++}$ \\  $\bXi^{-1}  \succ \bA \bM^{-1} \bA^\top + \bOmega^{-1}$ } \\ 
    \Xhline{1.2pt}  
     \end{tabular}  } 
\vskip 0.5em
\end{table}

These algorithms can be interpreted by alternating optimization of some cost function. For instance, \texttt{LAG-I} and \texttt{LAG-VII} stem from the alternating optimization of non-augmented Lagrangian $\cL(\bx,\ba,\bp)$. For example, both $\bx$- and $\ba$-updates of \texttt{LAG-I} come from
\[
\left\lfloor  \begin{array}{lll}
\bx^{k+1} & = & \arg \min_\bx  \cL 
(\bx,\ba^k,\bp^k) 
+  \frac{1}{2} \|\bx - \bx^k \|_{\bM}^2,  \\
\ba^{k+1} & =  & \arg \min_\ba  \cL 
(\bx^k, \ba, \bp^k) 
  +\frac{1}{2} \|\ba - \ba^k\|_{\bOmega}^2.
 \end{array} \right. 
 \]
 
\texttt{LAG-V} and \texttt{LAG-VI} are based on the augmented Lagrangian $\cL_\bXi(\bx,\ba,\bp)$. For example, the $\bx$- and $\ba$-updates of \texttt{LAG-V} are obtained by
\[
\left\lfloor  \begin{array}{lll}
\bx^{k+1} & = & \arg \min_\bx  \cL_\bXi 
(\bx,\ba^k,\bp^k) 
+  \frac{1}{2} \|\bx - \bx^k \|_{\bM}^2,  \\
  \ba^{k+1} & =  & \arg \min_\ba  \cL_\bXi 
(\bx^k, \ba, \bp^k) 
  +\frac{1}{2} \|\ba - \ba^k\|_{\bOmega}^2.
 \end{array} \right. 
 \]
 
The $\bx$- and $\ba$-updates of \texttt{LAG-II}, \texttt{LAG-III} and \texttt{LAG-IV} are the hybrid optimizations of both non-augmented and augmented forms. For example, the $\bx$-update of \texttt{LAG-IV} is from non-augmented, while the $\ba$-update is from augmented, i.e.,
\[
\left\lfloor  \begin{array}{lll}
\bx^{k+1} & = & \arg \min_\bx  \cL (\bx,\ba^{k},\bp^k) 
+  \frac{1}{2} \|\bx - \bx^k \|_{\bM}^2,  \\
\ba^{k+1} & =  & \arg \min_\ba  \cL_\bXi 
(\bx^{k+1}, \ba,  \bp^k) 
  +\frac{1}{2} \|\ba - \ba^k\|_{\bOmega}^2.  \\ 
 \end{array} \right. 
 \]
 
The preconditioning technique \cite[Sect. 4.3]{cp_2011} can be applied to the $\bx$-updates of \texttt{LAG-II,III,V,VI} and   $\ba$-updates of \texttt{LAG-IV,V,VI}, see \cite[Sect. 4.1]{fxue_gopt} for more details.

If $\bXi = \gamma \I_M$, \texttt{LAG-V}  and \texttt{LAG-VI}  reduce to \cite[Eqs.(19) and (20)]{fxue_gopt}. Their comparisons and connections to \cite[Algorithms 1 and 2]{shefi} have been discussed in \cite[Sect. 4.1]{fxue_gopt}.  In addition, the convergence condition of \cite[Eq.(19)]{fxue_gopt}, by \cite[Proposition 5.2 and Theorem 5.1]{shefi}, is $\bM \succ \gamma \bA^\top \bA$ and $\bOmega \succ \gamma \I_M$. Our analysis in Table 3 shows that this condition can be relaxed to $\bM \succ \gamma^2 \bA^\top \bOmega^{-1} \bA$, which is obviously milder than $\bM \succ \gamma \bA^\top \bA$ and $\bOmega \succ \gamma \I_M$.

Finally, note that the monotone operator $\cA$ represents the optimality condition of \eqref{lag}. Indeed, $\bcstar \in \zer \cA$ implies the KKT conditions:   $- \bA^\top \bp^\star \in \partial f(\bxstar)$, $\bp^\star \in \partial g(\ba^\star)$ and 
 $\ba^\star = \bA \bx^\star$, i.e. ${\bf 0} \in \partial f(\bxstar) +\bA^\top \partial g(\bA\bxstar)$. This is the reason for why all the Lagrangian-based schemes in Table 1 share the same $\cA$.
 
Another important observation is that $\cA$ bears a typical (diagonal) monotone + (off-diagonal) skew-symmetric structure:
\[
\begin{bmatrix}
\partial f & \bf 0  & \bA^\top \\
\bf 0 & \partial g  & -\I_{M}   \\
-\bA  & \I_{M} & \bf 0   
\end{bmatrix} = \underbrace{  \begin{bmatrix}
\partial f & \bf 0  & \bf 0 \\
\bf 0 & \partial g  & \bf 0   \\
\bf 0  & \bf 0 & \bf 0   
\end{bmatrix} }_\text{monotone}  +
\underbrace{  \begin{bmatrix}
\bf 0 & \bf 0  & \bA^\top \\
\bf 0 & \bf 0  & -\I_{M}   \\
-\bA  & \I_{M} & \bf 0   
\end{bmatrix} }_\text{skew}, 
\]
which has also been noticed in \cite{arias_2011,bredies_2017,plc_fixed}. This remark also applies to other classes of algorithms, see Sect. \ref{sec_pd} and \ref{sec_mix}.

\subsection{Connections to existing algorithms}
\label{sec_lag_eq}
\cite[Sect. 3.1]{fxue_gopt} discussed the connection of a special case of \texttt{LAG-I} to PDHG. We here show more connections.

\subsubsection{\texttt{LAG-I}: two forms of PDHG}
Letting  $\bu: = \begin{bmatrix}
\bx \\ \ba \end{bmatrix}$, $q(\bu):=f(\bx)+g(\ba)$, $\bU: = \begin{bmatrix} \bA & - \I_M \end{bmatrix}$, the Lagrangian \eqref{lag} is compactly given as 
\be  \label{e21}
\cL(\bu,\bp) = q(\bu) +\langle \bp | \bU\bu \rangle.
\ee 
This Lagrangian objective consists of the primal part of $q(\bu)$, the dual part of $0$, and their interplay represented by $\langle \bp | \bU\bu \rangle$.   \texttt{LAG-I} is equivalently written as
\[
\left\lfloor  \begin{array}{lll}
\bu^{k+1} &: = & \prox_q^\bR \big(  \bu^k -\bR^{-1} \bU^\top \bp^k  \big),  \\
 \bp^{k+1}  & := &    \bp^k + \bGamma \bU  
 \big(  2 \bu^{k+1}  -\bu^k  \big),
  \end{array} \right. 
\]
where $\bR = \begin{bmatrix} \bM & \bf 0  \\
\bf 0 & \bOmega \end{bmatrix}$. This is essentially a special case of \texttt{PDS-I} in Sect. \ref{sec_pds_algo}, where $\bu$-step is a primal update, $\bp$-step is a dual update. Compared to the commonly used primal-dual form  \eqref{pd},  \texttt{LAG-I} associated with the Lagrangian \eqref{lag} or \eqref{e21} adopts a different splitting strategy, which treats $\bu=(\bx,\ba)$ as primal variable and $\bp$ as dual, whereas \eqref{pd} treats $f(\bx)$ as primal and $g^*(\bp)$ as dual. 

The following proposition shows that under a certain condition, \texttt{LAG-I} can be simplified to the alternating updates between $f$ and $g^*$, which coincides with the splitting strategy of \eqref{pd}. This result also extends the discussion in \cite[Sect. 3.1]{fxue_gopt} to general proximal metrics, and thus, the proof is omitted.
\begin{proposition}  \label{p_lag_1}
Given  \texttt{LAG-I}, then, the following hold.

{\rm (i)} \texttt{LAG-I}  is equivalent to
\be \label{lag_1_eq}
\left\lfloor  \begin{array}{lll}
\bx^{k+1} &: = & \prox_f^\bM \big(  \bx^k -\bM^{-1} \bA^\top \bp^k  \big),  \\
  \bs^{k+1} &: =  & \prox_{g^*}^{\bOmega^{-1}}  
  ( \bOmega \ba^k +  \bp^k  ),       \\ 
\ba^{k+1}  & := &  \ba^k +\bOmega^{-1}  (\bp^k -
  \bs^{k+1} ), \\
\bp^{k+1}  & := &    \bp^k + \bXi  
 \big(  \bA (2 \bx^{k+1}  -\bx^k ) - (2\ba^ {k+1} -\ba^k)  \big). \end{array} \right. 
 \ee 

{\rm (ii)} If $ \bOmega = 2\bXi$,  $\bs^k = \bp^k$, \eqref{lag_1_eq} reduces to
\be \label{xyxz}
\left\lfloor  \begin{array}{l}
\bx^{k+1} := \prox_f^\bM \big( \bx^k - \bM^{-1} \bA^\top \bp^k \big),  \\
\bp^{k+2} := \prox_{ g^* }^{ \bOmega^{-1}} \big( \bp^k + 
\bOmega \bA (2 \bx^{k+1}  -\bx^k ) \big).  
\end{array}   \right.
\ee 
\end{proposition}

Observe that the scheme \eqref{xyxz} is essentially \texttt{PDS-I}---a generalized version of PDHG \cite[Eq. (8)]{fxue_gopt}, which will be discussed in Sect. \ref{sec_pd}. 

If one chooses $\bOmega = \bXi$ in \texttt{LAG-I} (which violates the convergence condition), and $\bs^k = \bp^k$, then combining the updates of $\bs$, $\ba$ and $\bp$ in \eqref{lag_1_eq},   we obtain $ \bp^{k+2} = \prox_{g^* }^{\bXi^{-1} }  \big( \bp^{k+1} +  \bXi \bA (2 \bx^{k+1}  -\bx^k )  \big)  $. 
Thus, \texttt{LAG-I} becomes
\[
\left\lfloor  \begin{array}{l}
\bx^{k+1} := \prox_f^\bM \big( \bx^k - \bM^{-1} \bA^\top \bp^k \big), \\
\bp^{k+1} := \prox_{g^*}^{\bXi^{-1} } \big( \bp^k + \bXi \bA  (2\bx^{k} - \bx^{k-1}) \big). 
\end{array}   \right.
\] 
This is a  PDHG-like algorithm, but with illogical and weird update (noting that $\bp^{k+1}$ is obtained without using $\bx^{k+1}$). It is not guaranteed to converge, due to the unreasonable assumption $\bOmega = \bXi$.

\subsubsection{\texttt{LAG-V}: semi-implicit Arrow-Hurwicz scheme}
We now show that \texttt{LAG-V} is essentially an instance of the classical semi-implicit Arrow-Hurwicz scheme. 

Let $\bu:= \begin{bmatrix}
\bx \\ \ba  \end{bmatrix} $, $q(\bu):=f(\bx)+g(\ba)$, $\bU := \begin{bmatrix}
\bA & - \I_M \end{bmatrix}$, the augmented Lagrangian \eqref{aug_lag} is compactly written as
\[
\cL_\bGamma(\bu, \bp) = q(\bu) + \langle \bp | \bU\bu \rangle 
+ \frac{1}{2} \|\bU \bu \|_\bGamma^2.
\]
With the variable metrics $\bR$ and $\bGamma$, the {\it semi-implicit} Arrow-Hurwicz scheme is given by
\[
 \left\lfloor  \begin{array}{lll}
\bu^{k+1 } & :\in  & \bu^k - \bR^{-1} \big( \partial q(\bu^{k+1})+\bU^\top \bp^k +\bU^\top \bGamma \bU \bu^{k+1} \big),  \\
\bp^{k+1} &:= & \bp^k + \bGamma \bU \bu^{k+1},
   \end{array}   \right.
\]
i.e., 
\be \label{x1}
 \left\lfloor  \begin{array}{lll}
\bu^{k+1 } & :=  & \prox_q^{\bR+\bU^\top \bGamma \bU} \big( (\bR+\bU^\top \bGamma \bU)^{-1} (\bR \bu^k - \bU^\top \bp^k ) \big),  \\
\bp^{k+1} &:= & \bp^k + \bGamma \bU \bu^{k+1}.
   \end{array}   \right.
\ee
The equivalent PPA form is given as
\[
\begin{bmatrix}
\bf 0 \\ \bf 0 \end{bmatrix} \in 
\begin{bmatrix}
\partial q & \bU^\top \\  -\bU  & \bf 0 \end{bmatrix} 
\begin{bmatrix}
\bu^{k+1} \\ \bp^{k+1} \end{bmatrix} +
\begin{bmatrix}
\bR & \bf 0 \\  \bf 0  & \bGamma^{-1} \end{bmatrix} 
\begin{bmatrix}
\bu^{k+1}- \bu^{k} \\ \bp^{k+1} - \bp^{k} \end{bmatrix}, 
\]
for which it is easy to show the convergence.

Furthermore, if one chooses $\bR  = \begin{bmatrix}
\bM & \bA^\top \bGamma \\ \bGamma \bA & \bOmega \end{bmatrix}$, such that $\bR +\bU^\top\bGamma\bU = \begin{bmatrix}
\bM +\bA^\top \bGamma \bA & \bf 0 \\
 \bf 0 & \bOmega+\bGamma \end{bmatrix}$, which  makes $\bx$ and $\ba$ to be fully decoupled, the Arrow-Hurwicz  scheme \eqref{x1} can be split into $(\bx, \ba,\bp)$:
\[
 \left\lfloor  \begin{array}{lll}
\bx^{k+1 } & :=  & \prox_f^{\bM+\bA^\top \bGamma \bA} \big( \bx^k - (\bM+\bA^\top \bGamma \bA)^{-1} ( \bA^\top\bGamma\bA   \bx^k - \bA^\top \bGamma \ba^k + \bA^\top \bp^k ) \big),  \\
\ba^{k+1 } & :=  & \prox_g^{\bOmega+\bGamma} \big( \ba^k - (\bOmega+\bGamma)^{-1} ( -\bGamma \bA  \bx^k +  \bOmega \ba^k - \bp^k ) \big),  \\
\bp^{k+1} &:= & \bp^k + \bGamma (\bA \bx^{k+1}-\ba^{k+1} )  ,
   \end{array}   \right.
\]
which is exactly \texttt{LAG-V}. The convergence condition (as shown in Table 3) follows from $\bR\succ \bf 0$.

\subsubsection{\texttt{LAG-VI}: ADMM and PDHG}
\texttt{LAG-VI} is essentially a proximal ADMM with proximal metrics $\bM$ and $\bOmega$. We now show the connection of \texttt{LAG-VI} to PDHG.
 
\begin{proposition}  \label{p_lag_2}
Given \texttt{LAG-VI},   the following hold:

{\rm (i)} If $\bOmega = \bf 0$, \texttt{LAG-VI} is equivalent to
\be \label{x78}
\left\lfloor  \begin{array}{l}
\bx^{k+1} :=   \prox_{f}^{ \bM+ \bA^\top \bXi \bA }\big( \bx^k - ( \bM + \bA^\top\bXi \bA)^{-1} \bA^\top  (2 \bp^k - \bp^{k-1}  ) \big),  \\
\bp^{k+1} := \prox_{ g^* }^{ \bXi^{-1}} \big( \bXi \bA 
\bx^{k+1} + \bp^k \big),  
\end{array}   \right.
\ee

{\rm (ii)} If $\bM = \frac{1}{\tau}\I_N - \gamma \bA^\top \bA$,  $\bOmega = \bf 0$,  $\bGamma = \gamma \I_M$, \texttt{LAG-VI} reduces to the PDHG \cite[Eq.(9)]{fxue_gopt}:
\[
\left\lfloor  \begin{array}{l}
\bx^{k+1} :=   \prox_{\tau f} \big( \bx^k - \tau \bA^\top  (2 \bp^k - \bp^{k-1}  ) \big),  \\
\bp^{k+1} := \prox_{\gamma g^* }  \big( \bp^k  + \gamma  \bA  \bx^{k+1} \big),  
\end{array}   \right.
\]

\end{proposition}

\begin{proof}
(i) If $\bOmega = \bf 0$, similar to Proposition \ref{p_lag_1}--(i),  \texttt{LAG-VI}  is equivalent to
\[
\left\lfloor  \begin{array}{lll}
\bx^{k+1} & := &  \prox_{f}^{ \bM+ \bA^\top \bXi \bA }\big( ( \bM + \bA^\top\bXi \bA)^{-1} ( 
\bM \bx^k  +\bA^\top  \bXi \ba^{k} - \bA^\top \bp^k  ) \big),  \\
  \bs^{k+1}   &: =  & \prox_{g^*}^{\bXi^{-1}}  \big(
   \bXi  \bA \bx^{k+1} +\bp^k ) \big),    \\ 
  \ba^{k+1}   &: =  & \bA\bx^{k+1} +\bXi^{-1} (\bp^k - \bs^{k+1} ),  \\ 
  \bp^{k+1}  & := & \bp^k + \bXi  \big(  \bA \bx^{k+1} -
 \ba^ {k+1}  \big), 
 \end{array} \right.  
\] 
which yields that $\bs^k = \bp^k$. Substituting $\ba^k = \bA\bx^{k} +\bXi^{-1} (\bp^{k-1} - \bp^{k} ) $ into $\bx$-update completes the proof.

\vskip.1cm
(ii) clear.
\end{proof}

Observe that the scheme \eqref{x78} is essentially a generalized version of PDHG \cite[Eq.(9)]{fxue_gopt}, the corresponding   PPA form is given as
\[
 \begin{bmatrix}
\bf 0 \\ \bf 0   \end{bmatrix}   \in 
\begin{bmatrix}
\partial f &   \bA^\top   \\
-\bA & \partial g^* \end{bmatrix}   \begin{bmatrix}
\bx^{k+1 } \\ \bp^{k }  
\end{bmatrix}    +   \begin{bmatrix}
\bM + \bA^\top \bXi \bA  &  \bA^\top  \\
 \bA  &  \bXi^{-1} \end{bmatrix}    \begin{bmatrix}
\bx^{k+1 } - \bx^k \\     \bp^{k }   - \bp ^{k-1}  
\end{bmatrix}.
\]  

We consider \texttt{LAG-V} as a comparison with  \texttt{LAG-VI}. If one chooses $\bOmega = \bf 0 $ (which violates the convergence condition), following the similar steps of Proposition \ref{p_lag_2}, \texttt{LAG-V} becomes
\[
\left\lfloor  \begin{array}{l}
\bx^{k+1} :=   \prox_{f}^{ \bM+ \bA^\top \bXi \bA }\big( \bx^k - ( \bM + \bA^\top\bXi \bA)^{-1} \bA^\top  (2 \bp^k - \bp^{k-1}  ) \big),  \\
\bp^{k+1} := \prox_{ g^* }^{ \bXi^{-1}} \big( \bXi \bA 
\bx^{k} + \bp^k \big),  
\end{array}   \right.
\]
where  the $\bp$-update is illogical and weird (noting that $\bp^{k+1}$ is computed {\it without} using $\bx^{k+1}$). It is not guaranteed to converge, due to the unreasonable assumption $\bOmega = \bf 0$.

\subsection{The generalized Bregman distance}
\label{sec_lag_bregman}
We will use the PPA interpretations to show that {\it the objective value that the Lagrangian schemes in Table 1 try to minimize is essentially  an instance of generalized Bregman distance associated with $f(\bx)+g(\ba)$}.

First, we define a quantity\footnote{For any pair of $(\bc,\bc')$, the quantity of $\bPi(\bc,\bc')$ is generally only a difference, but not a distance, since it is not guaranteed to be non-negative.}:
\[
\bPi(\bc, \bc') := \cL(\bx,\ba,\bp') - \cL(\bx',\ba',\bp),
\]
where $\cL(\bx,\ba, \bp)$ is given by  \eqref{lag}. Given the schemes in Table 1, the following lemma presents a key inequality,  which directly connects $\bPi(\bctilde^k,\bc)$ to the  metric $\cQ$.
\begin{lemma} \label{l_gap_lag}
Given the Lagrangian $\cL(\bx,\ba,\bp)$ as \eqref{lag}, consider all the Lagrangian-based schemes listed in Table 1, where $\bctilde^k = (\bxtilde^k,\batilde^k,\bptilde^k) $ denotes the proximal output, when the schemes are interpreted by the PPA (shown in Table 2). Then,  the following holds, $\forall \bc =(\bx,\ba,\bp) \in \R^N \times \R^M \times \R^M$:

{\rm (i)} $\bPi(\bctilde^k, \bc) \le \big  \langle \cQ (\bctilde^k - \bc^k)  \big| \bc -\bctilde^k  \big \rangle$,

{\rm (ii)} $\bPi  \big( \frac{1}{k}  \sum_{i=0}^{k-1} \bctilde^i, \bc  \big)  \le \frac{1}{2k} \big\| \bc^0 - \bc \big\|_\cS^2$.
\end{lemma}
\begin{proof}
(i) First, note that the proximal step of all the Lagrangian-based schemes listed in Table 1 can be written as:
\[
\begin{bmatrix}
\bf 0 \\ \bf 0 \\ \bf 0 \end{bmatrix} \in 
\begin{bmatrix}
\partial f & \bf 0 & \bA^\top \\
\bf 0 & \partial g &  -\I_M \\
-\bA & \I_M & \bf 0  \end{bmatrix}
\begin{bmatrix}
\bxtilde^k \\ \batilde^k \\ \bptilde^k \end{bmatrix} 
+ \begin{bmatrix}
\text{---}\bQ_1\text{---} \\ 
\text{---}\bQ_2\text{---} \\ 
\text{---}\bQ_3\text{---} \end{bmatrix} 
(\bctilde^k - \bc^k ), 
\]
which is:
\be \label{q3}
\left\lfloor \begin{array}{llll}
\bf 0  & \in  & \partial f(\bxtilde^k) +\bA^\top \bptilde^k  + \bQ_1 (\bctilde^k - \bc^k), \\ 
\bf 0  & \in  & \partial g(\batilde^k)  - \bptilde^k  + \bQ_2 (\bctilde^k - \bc^k), \\ 
\bf 0  & =  & -\bA\bxtilde^k +\batilde^k +  \bQ_3 (\bctilde^k - \bc^k). 
 \end{array}  \right.
\ee
Then, by convexity of $f$ and $g$, we develop:
\begin{eqnarray}
f(\bx) & \ge & f(\bxtilde^k) + \langle \partial f(\bxtilde^k) | \bx - \bxtilde^k \rangle 
\nonumber \\
& = &  f(\bxtilde^k) - \langle \bA^\top \bptilde^k  | \bx - \bxtilde^k \rangle
- \langle \bQ_1(\bctilde^k - \bc^k)  | \bx - \bxtilde^k \rangle,  \quad \text{by \eqref{q3}} \nonumber
\end{eqnarray}
and
\begin{eqnarray}
g(\ba) & \ge & g(\batilde^k) + \langle \partial g(\batilde^k) | \ba - \batilde^k \rangle 
\nonumber \\
& = &  g(\batilde^k) + \langle  \bptilde^k  | \ba - \batilde^k \rangle
- \langle \bQ_2(\bctilde^k - \bc^k)  | \ba - \batilde^k \rangle.
\quad \text{by \eqref{q3}} \nonumber
\end{eqnarray}
Summing up both inequalities yields
\begin{eqnarray}
f(\bx) +g(\ba) - f(\bxtilde^k) - g(\batilde^k) 
&\ge & - \langle \bA^\top \bptilde^k  | \bx - \bxtilde^k \rangle
- \langle \bQ_1(\bctilde^k - \bc^k)  | \bx - \bxtilde^k \rangle
\nonumber \\ 
& + & \langle  \bptilde^k  | \ba - \batilde^k \rangle
- \langle \bQ_2(\bctilde^k - \bc^k)  | \ba - \batilde^k \rangle.
\nonumber
\end{eqnarray}
Finally, we obtain
\begin{eqnarray}
\bPi(\bctilde^k, \bc) &= & \cL(\bxtilde^k,\batilde^k,\bp) - \cL(\bx,\ba,\bptilde^k)
\nonumber \\
& =& f(\bxtilde^k) + g(\batilde^k)  - f(\bx) - g(\ba) 
+\langle \bp | \bA\bxtilde^k - \batilde^k \rangle
-\langle \bptilde^k | \bA\bx - \ba \rangle
\nonumber \\
&\le &   \langle \bQ_1(\bctilde^k - \bc^k)  | \bx - \bxtilde^k \rangle +  \langle \bQ_2(\bctilde^k - \bc^k)  | \ba - \batilde^k \rangle
\nonumber \\ 
& - & \langle  \bptilde^k  | \ba - \batilde^k \rangle
+ \langle \bA^\top \bptilde^k  | \bx - \bxtilde^k \rangle
+ \langle \bp | \bA\bxtilde^k - \batilde^k \rangle
-\langle \bptilde^k | \bA\bx - \ba \rangle
\nonumber \\
& = &   \langle \bQ_1(\bctilde^k - \bc^k)  | \bx - \bxtilde^k \rangle +  \langle \bQ_2(\bctilde^k - \bc^k)  | \ba - \batilde^k \rangle +
\langle \bA\bxtilde^k  - \batilde^k | \bp -\bptilde^k  \rangle
\nonumber \\
& = &   \langle \bQ_1(\bctilde^k - \bc^k)  | \bx - \bxtilde^k \rangle +  \langle \bQ_2(\bctilde^k - \bc^k)  | \ba - \batilde^k \rangle +
\langle \bQ_3(\bctilde^k-\bc^k) | \bp -\bptilde^k  \rangle
\  \text{by \eqref{q3}} 
\nonumber  \\ 
& = &   \langle \cQ (\bctilde^k - \bc^k)  |
\bc -\bctilde^k  \rangle. 
\nonumber 
\end{eqnarray}

(ii) By Lemma \ref{l_gap_lag}-(i) and Lemma \ref{l_gppa}-(i), we further develop:
\begin{eqnarray}
\bPi(\bctilde^i, \bc) &  \le & \big  \langle \cQ (\bctilde^i - \bc^i)  \big| \bc -\bctilde^i  \big \rangle
\nonumber \\
&=&  \frac{1}{2} \big\| \bc^i - \bc \big\|_\cS^2
- \frac{1}{2} \big\| \bc^{i+1} - \bc \big\|_\cS^2
- \frac{1}{2} \big \| \bc^i -  \bc^{i+1}  \big \|
_{\cM^{-\top} \cG \cM^{-1} }^2.
\nonumber 
\end{eqnarray}
Then, summing up from $i=0$ to $k-1$, we obtain
$\sum_{i=0}^{k-1} \bPi(\bctilde^i, \bc)   \le 
\frac{1}{2} \big\| \bc^0 - \bc \big\|_\cS^2$. 
Since $\bPi(\bctilde^i, \bc) $ is a convex function w.r.t. $\bctilde^i$ (by its definition), it yields that
$\frac{1}{k}\sum_{i=0}^{k-1} \bPi(\bctilde^i,\bc)  \ge  \bPi  \big(\frac{1}{k}  \sum_{i=0}^{k-1} \bctilde^i, \bc \big)$,
which completes the proof. 
\end{proof}

Noting that Lemma \ref{l_gap_lag} is valid for any $\bc \in \R^N \times \R^M \times \R^M$, $\bPi(\bctilde^k,\bc)$ is not a distance, since it may be negative. However,   $\bPi(\bc,\bcstar)$ with $\bcstar \in \zer \cA$ is essentially a particular instance of the generalized Bregman distance generated by $q(\bu):=f(\bx) +g(\ba)$ between any point $\bu= (\bx,  \ba)$ and a saddle point  $\bustar= (\bxstar, \bastar)$. More specifically, $0 \le D_q^\flat(\bu,\bu^\star) \le \bPi(\bc, \bcstar)  \le  D_q^\sharp (\bu,\bu^\star)$.  Indeed, the generalized Bregman distance is given as
\begin{eqnarray} \label{w2}
0 & \le & D_q^\flat(\bu,\bustar) =  q(\bu) - q(\bustar) +  \inf_{\bv \in \partial q(\bustar)} 
\langle \bv  |  \bustar - \bu \rangle 
\nonumber \\
&=&  f(\bx) - f(\bxstar) + g(\ba) - g(\bastar) 
+ \inf_{\bv \in \partial f(\bxstar)  }  \langle \bv | \bxstar - \bx  \rangle   + \inf_{\bt \in \partial g(\bastar)  } \langle \bt |  \bastar  - \ba \rangle 
\nonumber \\
&\le &  f(\bx) - f(\bxstar) + g(\ba) - g(\bastar) + \langle  \bA^\top \bpstar | \bx - \bxstar \rangle   - \langle \bpstar | \ba - \bastar \rangle 
\nonumber \\
&=&  f(\bx) - f(\bxstar)  + g(\ba) - g(\bastar)  +\langle  \bpstar | \bA \bx -   \ba  \rangle \ \text{(by $\bA\bxstar = \bastar$)}
\nonumber \\
&=&  f(\bx) - f(\bxstar)  + g(\ba) - g(\bastar)  +\langle  \bpstar | \bA \bx -   \ba  \rangle - \langle  \bp |\underbrace{  \bA \bxstar -   \bastar }_{= \bf 0} \rangle 
\nonumber \\
&=& \bPi(\bc, \bcstar) \le D_q^\sharp (\bu,\bustar), 
 \end{eqnarray}
where the first inequality, i.e., the non-negativity of $ D_q^\flat(\bu,\bustar)$, is due to the convexity of $q$.

Then, we obtain the  convergence rate of  $\bPi(\bc^k, \bcstar) $ in an ergodic sense. 
\begin{theorem} \label{t_bregman_lag}
For all the Lagrangian-based algorithms shown in Table 1, the generalized Bregman distance associated with $f(\bx)+g(\ba)$ between the ergodic point $\frac{1}{k}  \sum_{i=0}^{k-1} \bctilde^i$ and a saddle point $\bcstar \in \zer \cA$ has a rate of $\cO(1/k)$:
\[
0\le \bPi  \bigg( \frac{1}{k}  \sum_{i=0}^{k-1} \bctilde^i, \bcstar \bigg)  \le \frac{1}{2k} \big\| \bc^0 - \bcstar \big\|_\cS^2,
\]
where $\{\bctilde^i\}_{i\in\N}$ and $\cS$ are defined in Lemma \ref{l_gap_lag} and \ref{l_gppa}.
\end{theorem}
\begin{proof}
The first inequality (i.e. non-negativity) follows from \eqref{w2}. The second inequality is concluded by simply taking $\bc = \bcstar \in \zer \cA$ in Lemma \ref{l_gap_lag}-(ii). 
\end{proof}

One can write down the specific form of $\|\bc^0-\bcstar\|_\cS$ for each Lagrangian algorithm, according to the associated $\cS$. In particular, for \texttt{LAG-I} and \texttt{LAG-II}, we have: $\bPi  \big( \frac{1}{k}  \sum_{i=1}^{k} \bc^i, \bcstar  \big)  \le  \frac{1}{2k} \big\| \bc^0 - \bcstar \big\|_\cQ^2$ (noting that $\bctilde^i = \bc^{i+1}$ due to $\cM=\cI$).
\begin{remark} [Degenerate case of Bregman distance] 
The generalized Bregman distance  $D_q^\flat (\bu,\bustar)$ or $ D_q^\sharp (\bu,\bustar)$ does not always reflect or control the distance between any point $\bu=(\bx,\ba)$ and a saddle point $\bustar=(\bxstar, \bastar)$, particularly for the non-strictly convex case of $q$. Consider a degenerate case of  linear functional, when  $f(\bx)=  \langle\bx |   \bv \rangle$ with a constant vector $\bv$ and $g(\ba)=0$. Then, the Bregman distance between any two points $\bu$ and $\bu'$ is always 0, since $ D_q^\flat (\bu,\bu')  =  D_q^\sharp (\bu,\bu') = D_f(\bx,\bx') = \langle \bx | \bv\rangle - \langle \bx' | \bv\rangle   -
\langle \bv | \bx-\bx'  \rangle =0 $. It implies that the Bregman distance loses the control of the distance between $(\bx,\ba)$ and $(\bx',\ba')$. Theorem \ref{t_bregman_lag} becomes more informative,  when the involved functions $f$ and $g$ are strictly convex.
\end{remark}

\subsection{The ergodic primal-dual gap}
\label{sec_lag_gap}
Considering $\cL(\bx,\ba,\bp)$ given as \eqref{lag}, for given sets $B_1 \subset \R^N$, $B_2 \subset \R^M$ and $B_3 \subset \R^M$, we introduce a {\it primal-dual gap} function restricted to $B_1 \times B_2 \times B_3$ 
  \cite[Eq.(2.6)]{bot_2014}, \cite[Sect. 3.1]{cp_2011}, \cite[Eq.(2.14)]{nem}:
\be \label{gap}
\bPsi_{B_1 \times B_2 \times B_3} (\bc) = \sup_{\bp' \in B_3} \cL(\bx,\ba,\bp') - 
\inf_{(\bx',\ba')\in B_1\times B_2} \cL(\bx',\ba',\bp),
\ee 
from which also follows that  $\bPsi_{B_1 \times B_2 \times B_3} (\bc)
= \sup_{\bc' \in B_1 \times B_2 \times B_3} \bPi (\bc,\bc') $.

\begin{corollary} \label{c_gap_lag}
Under the conditions of Theorem \ref{t_bregman_lag}, if the set    $ B_1 \times B_2 \times B_3$ is bounded, the primal-dual gap defined as \eqref{gap} has the upper bound:
\be \label{e3}
\bPsi_{B_1 \times B_2 \times B_3}   \bigg( \frac{1}{k}  \sum_{i=0}^{k-1} \bctilde^i  \bigg)  \le \frac{1}{2k}
\sup_{\bc \in B_1 \times B_2 \times B_3} \big\| \bc^0 - \bc \big\|_\cS^2.
\ee 
Furthermore, $\bPsi_{B_1 \times B_2 \times B_3} (\frac{1}{k}  \sum_{i=0}^{k-1} \bctilde^i  ) \ge 0$, if the set $B_1 \times B_2 \times B_3$ contains a saddle point $\bcstar = (\bxstar,\bastar,\bpstar) \in \zer \cA$.
\end{corollary}
\begin{proof}
Since Lemma \ref{l_gap_lag}-(ii) is valid for any $(\bx,\ba,\bp)\in \R^N \times \R^M \times \R^M$,  passing to the  supremum and infimum over $(\bx,\ba) \in B_1 \times  B_2 $ and $\bp \in B_3$ yields \eqref{e3}. The non-negativity of 
$\bPsi_{B_1 \times B_2 \times B_3}$ follows from Theorem \ref{t_bregman_lag}, since $\bPsi_{B_1 \times B_2 \times B_3}$  $ \big( \frac{1}{k}  \sum_{i=0}^{k-1} \bctilde^i \big) \ge  \bPi  \big( \frac{1}{k}  \sum_{i=0}^{k-1} \bctilde^i, \bcstar \big) \ge 0$, provided that $B_1 \times B_2 \times B_3$ contains a saddle point $\bcstar \in \zer \cA$.
\end{proof}

\begin{remark}
The ergodic primal-dual gap for specific algorithms has been given in \cite{cp_2011,cp_2016,nem}. Our analysis of the primal-dual gap is general, easy and clear,  compared to the original complicated case studies of specific algorithms, e.g. \cite[Theorem 1-(b)]{cp_2011}, \cite[Theorem 9-(b)]{bot_2014} and \cite[Theorem 2.1-(d)]{bot_jmiv_2014}. All the results presented in Sect. \ref{sec_lag_bregman} and \ref{sec_lag_gap} are valid for all Lagrangian-based algorithms with  the same monotone operator $\cA$, not limited to the listed ones. More importantly, this observation also applies to other classes of algorithms, see Sect. \ref{sec_pds_gap} and \ref{sec_mix_gap}.
\end{remark}

\begin{remark} \label{r_gap_lag}
By Theorem \ref{t_gppa}, the multiplier $\{ \tilde{\bp}^k\}_{k \in \N}$ converges and therefore lies in some (unknown) bounded set $B_3 \subset \R^M$. If $\dom f$ and $\dom g$ are bounded,   Corollary \ref{c_gap_lag} could lead to an interesting result: {\it the sequence of the objective value of dual to \eqref{primal} taken at the ergodic averaging point $\frac{1}{k} \sum_{i=0}^{k-1} \bctilde^i$ converges at the rate of $\cO(1/k)$}, namely, it holds that:
\be   \label{gap_lag}
f^* \bigg( - \bA^\top \Big( \frac{1}{k}  \sum_{i=0}^{k-1} \tilde{\bp}^i \Big) \bigg) + g^* \bigg( \frac{1}{k}  \sum_{i=0}^{k-1} \tilde{\bp}^i  \bigg)    - f^*( - \bA^\top  \bp^\star)
-g^*(\bpstar)   \le  C/k,
\ee
for some constant $C$.

Indeed, if $\dom f$ and $\dom g$ are bounded, one can choose  $B_1 = \dom f$ and $B_2=\dom g$.  Since the sequence $\{ \tilde{\bp}^k\}_{k\in\N}$ lies in $B_3$, and thus, $ \frac{1}{k}  \sum_{i=0}^{k-1} \tilde{\bp}^i  \in B_3$, $\bpstar \in B_3$. Denoting the ergodic averaging point by
 $\bchat^k = \frac{1}{k}  \sum_{i=0}^{k-1} \tilde{\bc}^i$ ($\bxhat^k$, $\bahat^k$ and $\bphat^k$ are defined similarly),
using Fenchel-Young inequality \cite[Proposition 13.15]{plc_book}, we develop
\begin{eqnarray}
&& \bPsi_{B_1 \times B_2 \times B_3}   \big( \bchat^k \big)  
\nonumber \\ 
&=& \sup_{\bp' \in B_3} \cL \big(\bxhat^k, \bahat^k, \bp' \big) -  \inf_{(\bx',\ba') \in B_1 \times B_2} \cL \big(
\bx', \ba', \bphat^k \big) 
\nonumber \\ 
&=& \sup_{\bp' \in B_3} q\big( \buhat^k \big) 
+\big\langle \bp' \big| \bU \buhat^k \big\rangle
 -   \inf_{ \bu'   \in B_1 \times B_2} \big( q(\bu') +
\big\langle \bphat^k \big| \bU\bu'  \big\rangle \big) 
 \quad \text{by \eqref{e21}}
\nonumber \\ 
& \ge &  q\big( \buhat^k \big) 
+\big\langle \bU^\top \bp^\star \big| \buhat^k  \big\rangle
+ q^* \big( - \bU^\top \bphat^k \big) 
\nonumber \\ 
& \ge &  -q^*( - \bU^\top  \bp^\star) 
+ q^* \big( - \bU^\top\bphat^k  \big) ,
\nonumber 
\end{eqnarray} 
which, combining with Corollary \ref{c_gap_lag}, yields
\[
 q^* \big( - \bU^\top\bphat^k  \big) 
 -q^*( - \bU^\top  \bp^\star) 
 \le \frac{1}{2k}
\sup_{\bc \in B_1 \times B_2 \times B_3} \big\| \bc^0 - \bc \big\|_\cS^2.
\]
By the definitions of $q$ and $\bU$ of \eqref{e21},  $q^*(-\bU^\top \bp) = f^*(-\bA^\top \bp) +g^*(\bp)$, which exactly coincides with the dual of \eqref{primal}, which is given as \eqref{dual}. The conclusion \eqref{gap_lag} is reached. 

On the other hand, note that  $q^*(-\bU^\top \bp) = \sup_{\bu \in B_1 \times B_2} \big(- q(\bu) - \langle \bp| \bU \bu \rangle\big)  = - \inf_{\bu \in B_1 \times B_2} \big( q(\bu)
+ \langle \bp| \bU \bu \rangle \big)  = - \min_\bu \cL(\bu, \bp)$, and thus, the saddle-point problem of $\cL(\bu,\bp)$ \eqref{e21} becomes
\[
\max_\bp \min_\bu \cL(\bu,\bp) = \max_\bp (-q^* (-\bU^\top p)) = -\min_\bp q^*(-\bU^\top \bp),
\]
which is essentially the minimization problem of the dual $q^*(-\bU^\top \cdot)$. 

Finally, we stress that the convergence rate of $\cO(1/k)$ of the dual value holds for all the the Lagrangian-based algorithms shown in Table 1.  However, as contrary to Remark \ref{r_gap_pds}, it is difficult to obtain an {\it a priori} estimate of the constant $C$, since the bounded set $B_3$ is unknown in practice. 
\end{remark}

\subsection{Reductions of some Lagrangian schemes}
\label{sec_red_lag}
\subsubsection{\texttt{LAG-I} and \texttt{LAG-II}}
Table 2 shows that \texttt{LAG-I} and \texttt{LAG-II} can be expressed as a standard PPA \eqref{ppa} with $\cM=\cI$. Both of them can be reduced to a simple resolvent by Corollary \ref{c_gppa}-(iii):
\be \label{aa}
\bv^{k+1} :=  \big(\cI +  \cQ^{-\frac{1}{2} } \circ \cA \circ 
\cQ^{-\frac{1}{2} } \big)^{-1}   \bv^k  ,
\ee
where $\bv^k = \cQ^{\frac{1}{2} }   \bc^k$, $\cQ$ is specified in Table 2 for \texttt{LAG-I} or \texttt{II}.

\texttt{LAG-II} deserves particular attention, since the corresponding metric $\cQ$ is allowed to be degenerate.

\paragraph{Low degeneracy of \texttt{LAG-II}}
Notice that $\bM=\bf 0$ is allowed for \texttt{LAG-II}, which  becomes
\[
 \left\lfloor  \begin{array}{lll}
\bx^{k+1} & := &  \arg \min_\bx f(\bx) + \frac{1}{2} 
\big\| \bA \bx -  \ba^{k} + \bXi^{-1}  \bp^k  \big\|^2_\bXi,  \\
\bp^{k+1}  & = & \bp^k + \bXi  \big(  \bA  \bx^{k+1} - \ba^ {k}  \big) , \\ 
  \ba^{k+1} & =  & \prox_{g}^{\bOmega } \big(   \ba^k +   \bOmega^{-1}  (2 \bp^{k+1} - \bp^k) \big)  .
 \end{array} \right.
 \]
Now, the  metric $\cQ$ is degenerate (i.e. positive {\it semi-}positive) with $\rank \cQ=2M < \dimsf(\R^N\times \R^M\times \R^M)=N+2M$. The rank-deficiency of $\cQ$ shows that the variable $\bx$ is redundant that does not really take part in the iterations of \texttt{LAG-II}. We can reduce \texttt{LAG-II} based on the analysis of \cite{bredies_preprint}.

\begin{proposition} \label{p_red_lag}
\texttt{LAG-II} with $\bM=\bf 0$ and $\bOmega \succ \bXi$  can be expressed as the following resolvent:
\[
\bv^{k+1} = \big( \cI+   \cD 
 (\cL + \cK^\top \circ \partial f^* \circ \cK) \cD \big)^{-1} \bv^k, 
\]
where  $\cD = \begin{bmatrix}
\bOmega & -\I _M\\ -\I_M & \bXi^{-1} \end{bmatrix}^{-\frac{1}{2}}$, 
$\cL = \begin{bmatrix}
\partial g  & -\I_M  \\  \I_M & \bf 0  \end{bmatrix}$,
$\cK = \begin{bmatrix} \bf 0 &  -\bA^\top \end{bmatrix}$. Here, the variable $\bv$ is linked to $(\ba,\bp)$ in \texttt{LAG-II} via:  $\bv^k :=   \begin{bmatrix}
\bOmega & -\I _M\\ -\I_M & \bXi^{-1} \end{bmatrix}^{\frac{1}{2}}  \begin{bmatrix}
\ba^k \\ \bp^k \end{bmatrix} \in \R^M \times \R^M$. 
 \end{proposition}
\begin{proof}
$\cQ$ of \texttt{LAG-II} can be decomposed as: $\cQ =  \begin{bmatrix}
\bf 0  \\ \cD \end{bmatrix}  \begin{bmatrix}
\bf 0 & \cD^\top \end{bmatrix} $, where $\cD \cD^\top = \begin{bmatrix}
\bOmega & -\I _M\\ -\I_M & \bXi^{-1} \end{bmatrix}$. For simplicity, one can choose $\cD = \cD^\top = \begin{bmatrix}
\bOmega & -\I _M\\ -\I_M & \bXi^{-1} \end{bmatrix}^{\frac{1}{2}}$. 
The standard PPA form \eqref{ppa} becomes
\[
\bc^{k+1} =  \bigg(\cA + \begin{bmatrix}
\bf 0  \\ \cD \end{bmatrix}  \begin{bmatrix}
\bf 0 & \cD^\top \end{bmatrix} \bigg)^{-1} 
\begin{bmatrix}
\bf 0  \\ \cD \end{bmatrix}  \begin{bmatrix}
\bf 0 & \cD^\top \end{bmatrix} \bc^k. 
\]
 Let $\bv^k :=  \begin{bmatrix}
\bf 0 & \cD^\top \end{bmatrix} \bc^k =  \cD^\top  \begin{bmatrix}
\ba^k \\ \bp^k \end{bmatrix} \in \R^M\times \R^M$.  By \cite[Theorem 2.13]{bredies_preprint}, we obtain the reduced PPA: 
\begin{eqnarray}
\bv^{k+1} &=& \begin{bmatrix}
\bf 0 & \cD^\top \end{bmatrix}   \bigg( \cA +
\begin{bmatrix}
\bf 0 \\ \cD \end{bmatrix} 
\begin{bmatrix}
\bf 0 & \cD^\top \end{bmatrix}  \bigg)^{-1} 
\begin{bmatrix}
\bf 0 \\ \cD \end{bmatrix}   \bv^k
\nonumber \\
&=& \bigg( \cI + \bigg( \begin{bmatrix}
\bf 0 & \cD^\top \end{bmatrix}  \cA^{-1} 
\begin{bmatrix}
\bf 0 \\ \cD  \end{bmatrix}  \bigg)^{-1}\bigg)^{-1} \bv^k
:=\big(\cI + \cAtilde \big)^{-1} \bv^k.
\nonumber 
\end{eqnarray}
To evaluate $\cAtilde$, we rewrite $\cA = \begin{bmatrix}
\partial f & -\cK \\ \cK^\top  & \cL \end{bmatrix}$, where
$\cL = \begin{bmatrix}
\partial g  & -\I_M  \\  \I_M & \bf 0  \end{bmatrix}$,
$\cK = \begin{bmatrix}  \bf 0 & -\bA^\top \end{bmatrix}$. 
Then,  
\[
\cR = \cA^{-1} \begin{bmatrix}
\bf 0 \\ \cD  \end{bmatrix} =  \begin{bmatrix}
\partial f & -\cK \\ \cK^\top & \cL \end{bmatrix}^{-1} 
\begin{bmatrix} 
\bf 0 \\ \cD  \end{bmatrix} 
= \begin{bmatrix}
\cR_1 \\ \cR_2  \end{bmatrix},
\]
which yields the solution: $ \cR_2=  (\cL + \cK^\top \circ \partial f^*\circ \cK)^{-1} \cD$.  Thus,
\[
\cAtilde =\bigg( \begin{bmatrix}
\bf 0 & \cD^\top \end{bmatrix}  \cA^{-1} 
\begin{bmatrix}
\bf 0 \\ \cD  \end{bmatrix}  \bigg)^{-1} = 
(\begin{bmatrix}
\bf 0 & \cD^\top \end{bmatrix}  \cR)^{-1}
= ( \cD^\top   \cR_2)^{-1}, 
\]
Substituting $\cR_2$ into above  concludes the proof. \hfill 
\end{proof}

\paragraph{High degeneracy of \texttt{LAG-II}}
Furthermore, if $\bM=\bf 0$ and $\bOmega=\bXi = \I_M$, 
\texttt{LAG-II} becomes
\be \label{ew}
 \left\lfloor  \begin{array}{lll}
\bx^{k+1} & := & \arg \min_\bx f(\bx) + \frac{1}{2} 
\big\| \bA \bx -  \ba^{k} +    \bp^k  \big\|^2, \\
\bp^{k+1}  & = & \bp^k +   \bA  \bx^{k+1} - \ba^ {k} , \\ 
  \ba^{k+1} & =  & \prox_{g}  \big(   \ba^k +  2 \bp^{k+1} - \bp^k  \big)  . 
 \end{array} \right.
 \ee
Now, the corresponding  metric $\cQ$ is `more' degenerate with $\rank \cQ=M$. The following result shows that the active variable of \texttt{LAG-II} is actually $\ba^k-\bp^k$. 

\begin{proposition} \label{p_red_lag_2}
\texttt{LAG-II} with $\bM=\bf 0$  and $\bOmega=\bXi = \I_M$  can be expressed as
\[
\bv^{k+1} = \big( \cI+ ( \tilde{\cD}^\top (\cL + \cK^\top \circ \partial f^*\circ  \cK)^{-1} \tilde{ \cD}) ^{-1}  \big)^{-1} \bv^k, 
\]
where  $\tilde{ \cD} = \begin{bmatrix}
\I _M &  -\I_M   \end{bmatrix}^\top$, 
$\cL = \begin{bmatrix}
\partial g  & -\I_M  \\  \I_M & \bf 0  \end{bmatrix}$,
$\cK = \begin{bmatrix} \bf 0 &  -\bA^\top \end{bmatrix}$. Here, the variable $\bv$ is linked to $( \ba,\bp)$ in \texttt{LAG-II} via:  $\bv^k := \ba^k - \bp^k \in \R^M$. 
 \end{proposition}
\begin{proof}
In this case, $\cQ$ of \texttt{LAG-II} can be decomposed as: $\cQ = \cD \cD^\top$ where $\cD = \begin{bmatrix}
\bf 0 & \I_M & -\I_M \end{bmatrix}^\top $. 
The standard PPA form \eqref{ppa} becomes: 
$\bc^{k+1} =  \big(\cA + \cD \cD^\top  \big)^{-1} 
 \cD \cD^\top   \bc^k$. 
Finally, the proof is completed, by \cite[Theorem 2.13]{bredies_preprint} and the proof of Proposition \ref{p_red_lag}. 
\end{proof}

The active variable of \eqref{ew}  can also be identified without the degenerate PPA analysis, as shown below.

From \eqref{ew}, we have
\[
  \ba^{k+1} - \bp^{k+1}  =
   \prox_{g}  \big(   \ba^k +  2 \bp^{k+1} - \bp^k  \big)
   - \bp^k  -   \bA  \bx^{k+1} + \ba^ {k}.
\]
Denoting $\bv^k:=\ba^k-\bp^k$, it becomes
\begin{eqnarray}
  \bv^{k+1}  &  =&   \prox_{g}  \big(   \bv^k +  2 \bp^{k+1}  \big)+ \bv^k  -   \bA  \bx^{k+1} 
  \nonumber\\
  &=&  \prox_{g}  \big(   \bv^k +  2 \bp^k +  2 \bA  \bx^{k+1} -2 \ba^ {k}  \big)+ \bv^k  -   \bA  \bx^{k+1} 
  \nonumber\\
  &=&  \prox_{g}  \big(   2 \bA  \bx^{k+1} -  \bv^ {k}  \big)+ \bv^k  -   \bA  \bx^{k+1} 
  \nonumber\\
  &=&  \prox_{g}  \big(   2   \prox_{\bA \triangleright f}  ( \bv^ {k}  ) - \bv^k \big) + \bv^k  -     \prox_{\bA \triangleright f}  ( \bv^ {k}  )
   \nonumber\\
  &=& \Big( \prox_{g}  \circ \big(   2 \prox_{\bA \triangleright f}  - \cI  \big) + \cI  -   \prox_{\bA \triangleright f}\Big) (\bv^k)
  \nonumber\\
  &=& \Big( \prox_{g}  \circ \big(   2 (\cI - \prox_{f^* \circ \bA^\top})  - \cI  \big) +
  \prox_{f^* \circ \bA^\top}  \Big) (\bv^k)\quad
  \text{(by Lemma \ref{l_duality})}
  \nonumber\\
  &=& \Big( \prox_{g}  \circ \big(  \cI -2 \prox_{f^* \circ \bA^\top}   \big) +
  \prox_{f^* \circ \bA^\top}  \Big) (\bv^k)
  \nonumber
\end{eqnarray}
which shows that \eqref{ew} is essentially a DRS algorithm (see Eq.\eqref{drs}).

\subsubsection{\texttt{LAG-VI} and the related standard ADMM/DRS}
It seems more interesting to investigate the degenerate case of \texttt{LAG-VI}, which is a representative ADMM-type algorithm.

Table 2 shows that the corresponding PPA of \texttt{LAG-VI} has a non-trivial relaxation step (i.e. $\cM\ne \cI$). This also coincides with a pioneering work of \cite[Sect. 3]{hbs_siam_2012}. Due to the non-trivial relaxation, it is difficult to obtain the equivalent resolvent from this PPA interpretation. {\it Can \texttt{LAG-VI} be written in a standard PPA form without relaxation step?} To achieve this,  by changing variable of $\bp^k:=\bz^k-\bXi \ba^k$,  \texttt{LAG-VI} becomes (with a flipped update order of $\bx \rightarrow \bz \rightarrow \ba$)
\be \label{padmm_eq}
 \left\lfloor  \begin{array}{lll}
\bx^{k+1 } & :=  & \prox_{f}^{\bM+\bA^\top\bXi \bA} \big( (\bM+ \bA^\top \bXi \bA)^{-1}   ( \bM \bx^k + 2 \bA^\top\bXi \ba^k -  \bA^\top \bz^k ) \big),  \\
\bz^{k+1} &:= & \bz^k + \bXi (\bA \bx^{k+1}-\ba^{k} ) , \\
\ba^{k+1 } & :=  & \prox_{g}^{\bOmega+\bXi}  \big(  
(\bOmega+\bXi)^{-1} ( (\bOmega-\bXi)\ba^k + \bXi \bA  \bx^{k+1} +  \bz^k )  \big) .
   \end{array}   \right.
\ee
It is easy to verify that \eqref{padmm_eq} corresponds to the standard PPA form (i.e. $\cM=\cI$):
\[
\begin{bmatrix}
\bf 0 \\ \bf 0 \\ \bf 0 \end{bmatrix} \in 
\begin{bmatrix}
\partial f & -\bA^\top\bXi & \bA^\top \\
\bXi \bA & \partial g & -\I_M  \\
-\bA & \I_M & \bf 0 \end{bmatrix} 
\begin{bmatrix}
\bx^{k+1} \\ \ba^k  \\ \bz^{k+1} \end{bmatrix} 
+  \begin{bmatrix}
\bM   &  \bf 0  & \bf 0 \\
\bf 0   &  \bOmega   & \bf 0 \\
\bf 0   &  \bf 0  & \bXi^{-1} \end{bmatrix} 
\begin{bmatrix}
\bx^{k+1}-\bx^k \\ \ba^k-\ba^{k-1}  \\
 \bz^{k+1} - \bz^k  \end{bmatrix} .
\]

\paragraph{Non-degenerate case: proximal ADMM}
If $\bM \succ \bf 0$ and $\bOmega \succ \bf 0$,  
$\cQ$ is non-degenerate. By Corollary \ref{c_gppa}-(iii), the equivalent resolvent is given  as \eqref{aa}, where  $\cA$ and $\cQ$ are specified as above, $\bv^k$ is related to $(\bx,\ba,\bp)$ of \eqref{admm} and $(\bx,\ba,\bz)$ of \eqref{admm_eq} via: $\bv^k = \cQ^{\frac{1}{2} } \bc^k
=(\bM^{\frac{1}{2}}\bx^k, \bOmega^{\frac{1}{2}}\ba^{k-1},
\bXi^{-\frac{1}{2}}\bz^k)
=(\bM^{\frac{1}{2}}\bx^k, \bOmega^{\frac{1}{2}}\ba^{k-1},
\bXi^{-\frac{1}{2}}\bp^k + \bXi^{\frac{1}{2}}\ba^k)$ .

\paragraph{Degenerate case: standard ADMM}
If $\bM=\bf 0$,  $\bOmega=\bf 0$ and $\bGamma = \gamma \I_M$, \texttt{LAG-VI} boils down to a standard ADMM \cite[Eq.(4)]{fxue_gopt}:
\be \label{admm}
 \left\lfloor  \begin{array}{lll}
\bx^{k+1 } & :=  & \arg \min_\bx f(\bx) + \frac{\gamma}{2} 
\big\|  \bA \bx -  \ba^{k} + \frac{1}{\gamma}   \bp^k  \big\|^2, \\
\ba^{k+1 } & :=  & \prox_{g/\gamma}  \big(  \bA  \bx^{k+1} +  \frac{ 1}{\gamma}  \bp^k \big),  \\
\bp^{k+1} &:= & \bp^k + \gamma (\bA \bx^{k+1}-\ba^{k+1} ).  
   \end{array}   \right.
\ee
By the variable changing  of $\bp^k:=\bz^k-\gamma \ba^k$, \eqref{admm}  becomes (with a flipped update order of $\bx \rightarrow \bz \rightarrow \ba$)
\be \label{admm_eq}
 \left\lfloor  \begin{array}{lll}
\bx^{k+1 } & :=  & \arg \min_\bx f(\bx) + \frac{\gamma}{2} 
\big\| \bA \bx - 2 \ba^{k} + \frac{1}{\gamma}   \bz^k  \big\|^2,  \\
\bz^{k+1} &:= & \bz^k + \gamma (\bA \bx^{k+1}-\ba^{k} ),  \\
\ba^{k+1 } & :=  & \prox_{g/\gamma}  \big( \frac{1}{\gamma}   \bz^{k+1} \big)  . 
   \end{array}   \right.
\ee
It is easy to verify that \eqref{admm_eq} corresponds to the standard PPA form (i.e., $\cM=\cI$):
\be \label{ppa_admm}
\begin{bmatrix}
\bf 0 \\ \bf 0 \\ \bf 0 \end{bmatrix} \in 
\begin{bmatrix}
\partial f & -\gamma \bA^\top & \bA^\top \\
\gamma \bA & \partial g & -\I_M  \\
-\bA & \I_M & \bf 0 \end{bmatrix} 
\begin{bmatrix}
\bx^{k+1} \\ \ba^k  \\ \bz^{k+1} \end{bmatrix} 
+  \begin{bmatrix}
\bf 0   &  \bf 0  & \bf 0 \\
\bf 0   &  \bf 0  & \bf 0 \\
\bf 0   &  \bf 0  & \frac{1}{\gamma} \I_M \end{bmatrix} 
\begin{bmatrix}
\bx^{k+1}-\bx^k \\ \ba^k-\ba^{k-1}  \\
 \bz^{k+1} - \bz^k  \end{bmatrix} .
\ee

The degenerate (i.e., positive {\it semi-}definite) metric $\cQ$ indicates  the  redundancy of the variables $\bx$ and $\ba$. Based on the recent result of \cite[Theorem 2.13]{bredies_preprint}, the standard ADMM \eqref{admm} or \eqref{admm_eq} can be reduced to a simple resolvent.

\begin{theorem} \label{t_admm}
The ADMM scheme \eqref{admm_eq}, being equivalent to \eqref{admm},  can be expressed as
\[
\bv^{k+1} = (\cI+  \gamma \cK^\top \cL^{-1} \cK)^{-1} \bv^k, 
\]
where $\cL = \begin{bmatrix}
\partial f  & -\gamma \bA^\top \\   
\gamma \bA & \partial g  \end{bmatrix}$,
$\cK = \begin{bmatrix}  -\bA & \I_M \end{bmatrix}^\top$. Here, the variable $\bv$ is linked to $\bz$ in \eqref{admm_eq} and $(\ba,\bp)$ in \eqref{admm} via  $ \bv^k =\frac{1}{  \sqrt{\gamma} } \bz^k = \frac{1}{  \sqrt{\gamma} } \bp^k 
+\sqrt{\gamma} \ba^k \in \R^M $.
\end{theorem}
\begin{proof}
$\cQ$ in  \eqref{ppa_admm}  can be decomposed as $\cQ = \cD \cD^\top =\begin{bmatrix}
\bf 0 \\ \bf  0 \\ \frac{1}{\sqrt{\gamma}} \I_M \end{bmatrix}  \begin{bmatrix}
\bf 0 & \bf 0 & \frac{1}{\sqrt{\gamma}} \I_M \end{bmatrix} $. Then, the standard PPA form \eqref{ppa_admm} becomes $\bc^{k+1} =  (\cA +\cD \cD^\top)^{-1} \cD \cD^\top \bc^k$.  Let $\bv^k := \cD^\top \bc^k =   \frac{1}{\sqrt{\gamma}} \bz^k$.  Finally, the result can be obtained by \cite[Theorem 2.13]{bredies_preprint} and similar proof of Proposition \ref{p_red_lag}. 
\end{proof}

\paragraph{Connection to standard DRS}
It is well known that the ADMM scheme \eqref{admm} is equivalent to the standard DRS algorithm \cite{lions} applied to the dual problem of \eqref{primal} (see \cite[Eq.(2)]{dinh} for example): 
\be \label{dual}
\min_\bp f^*(-\bA^\top \bp) + g^*(\bp),
\ee
which reads as
\be \label{drs}
\bz^{k+1} := \bz^k -   J_{\gamma \cB_2}  (\bz^k) +  J_{\gamma \cB_1}  \big( 2   J_{\gamma \cB_2}( \bz^k)   -  \bz^k \big) ,
\ee
where  $\cB_1=(-\bA) \circ \partial f^* \circ (-\bA^\top)$, $\cB_2 = \partial g^*$,  $J_\cB$ denotes a resolvent of $\cB$: $J_{\cB} = (\cI +\cB)^{-1}$. The solution to \eqref{dual} is given as: $\bpstar = J_{\gamma \cB_2} (\bzstar) = 
\prox_{\gamma g^*} (\bzstar)$.

Let us first examine the equivalence between ADMM \eqref{admm} or \eqref{admm_eq} and DRS \eqref{drs}, though  this fact has long been recognized. By the development of \eqref{drs}:
\[
 \left\lfloor  \begin{array}{lll}
\bw^{k+1 } & :=  & J_{ \gamma \cB_1} \big( 2\bp^k - \bz^k  \big) , \\
\bz^{k+1} &:= & \bz^k + \bw^{k+1}-\bp^{k} , \\
\bp^{k+1 } & :=  & \prox_{ \gamma \cB_2}  \big(  \bz^{k+1} \big)  .
   \end{array}   \right.
\]
By the similar technique of \cite[Sect. 4.2]{cp_2011}, we obtain by duality that:
\[
 \left\lfloor  \begin{array}{lll}
\bx^{k+1 } & :=  & \arg\min_\bx f(\bx) + \frac{\gamma}{2} \big\|\bA \bx+\frac{1}{\gamma} ( 2\bp^k - \bz^k)  \big\|^2,  \\
\bw^{k+1 } & :=  & 2\bp^k -\bz^k +\gamma \bA \bx^{k+1} , \\
\bz^{k+1} &:= & \bz^k + \bw^{k+1}-\bp^{k} , \\
\ba^{k+1 } & :=  & \arg\min_\bu g(\ba) + \frac{\gamma}{2} \big\| \ba - \frac{1}{\gamma} \bz^{k+1}  \big\|^2 , \\
\bp^{k+1 } & := &  \bz^{k+1} -\gamma  \ba^{k+1} .
   \end{array}   \right.
\]
Finally, \eqref{admm} can be obtained by keeping $(\bx,\ba,\bp)$ and removing $(\bw,\bz)$; while \eqref{admm_eq} is from  keeping $(\bx,\ba,\bz)$ and removing $(\bw,\bp)$.

The above equivalence implies that Theorem \ref{t_admm} also applies to the DRS iteration \eqref{drs}. Recall that the equivalence between DRS \eqref{drs} and PPA was discussed in an early seminal work of \cite{drs_1992}, where the DRS \eqref{drs} was shown to be equivalent to a resolvent with an {\it implicit} expression of the associated maximally monotone operator (see \cite[Sect. 4]{drs_1992}). Here, Theorem \ref{t_admm} shows another equivalent resolvent of DRS \eqref{drs}, with an {\it explicit} form of the monotone operator. However, the equivalence or connection between both forms requires further study.

\section{Operator splitting based on primal-dual form}
\label{sec_pd}

\subsection{The PDS algorithms and their PPA interpretations}
\label{sec_pds_algo}
We then consider the alternating optimization of the primal-dual form \cite[Eq.(7)]{fxue_gopt}:
\be \label{pd}
\min_\bx \max_\bp \cL(\bx,\bp) :=
 f(\bx) + \bp^\top \bA\bx - g^*(\bp), 
\ee
which gives rise to the  PDS algorithms shown in Table 4. \texttt{PDS-I,II,V,VI} and \texttt{VII} can be found in \cite[Sect. 5]{fxue_gopt}. Tables 5--6 show their equivalent PPA forms, by noting that:
\begin{itemize}
\item \texttt{PDS-I} and \texttt{PDS-II} correspond  to symmetric $\cQ$ (without relaxation): the off-diagonal parts of both $\cQ$ have opposite signs, which results in the reverse update orders of $\bx$ and $\bp$; 

\item \texttt{PDS-III} and \texttt{PDS-V}  correspond to lower triangular $\cQ$; 

\item \texttt{PDS-IV}  and \texttt{PDS-VI} correspond to upper triangular $\cQ$;

\item \texttt{PDS-VII}  corresponds  to skew-symmetric $\cQ$.
\end{itemize}

\begin{table} [h!] \label{table_pds_algo}
\centering
\caption{The proposed PDS algorithms }
\vspace{-.5em}
\hspace*{-.2cm}
\resizebox{.97\columnwidth}{!} {
\begin{tabular}{|l|l|} 
    \Xhline{1.2pt}  
    name  & iterative scheme \\
\hline 
  \tabincell{l}{ \texttt{PDS-I} \\ \cite[Eq.(26)]{fxue_gopt} }
  & $  \left\lfloor  \begin{array}{l}
\bx^{k+1} := \prox_f^\bM \big( \bx^k - \bM^{-1} \bA^\top \bp^k \big) \\
\bp^{k+1} := \prox_{g^*}^{\bXi}  \big( \bp^k +  \bXi^{-1}  \bA (2\bx^{k+1} - \bx^k) \big)
\end{array}   \right. $ \\  
 \hline
   \tabincell{l}{ \texttt{PDS-II} \\ \cite[Eq.(29)]{fxue_gopt} }
  & $   \left\lfloor  \begin{array}{l}
\bp^{k+1} := \prox_{g^* }^{\bXi} ( \bp^k +  \bXi^{-1} \bA  \bx^{k}  ) \\
\bx^{k+1} := \prox_f^\bM \big(  \bx^k - \bM^{-1}  \bA^\top (2\bp^{k+1} -  \bp^{k}) \big)  
\end{array}   \right. $ \\
 \hline
 \texttt{PDS-III} & $   \left\lfloor  \begin{array}{lll}
\bp^{k+1 } &:= & \prox_{g^*}^\bXi \big(  \bp^k + \bXi^{-1} \bA  \bx^k \big)  \\
\bxtilde^{k} &:= &  \prox_{f }^{\bM}  ( \bx^k - \bM^{-1}  \bA^\top   \bp^{k +1} ) \\
\bx^{k+1} &: =& \bxtilde^k  - \bM^{-1} \bA^\top \bXi^{-1} \bA  (\bxtilde^k - \bx^k) - \bM^{-1} \bA^\top ( \bp^{k+1} - \bp^k)     \end{array} \right.  $\\
 \hline
 \texttt{PDS-IV}  & $   \left\lfloor  \begin{array}{lll}
\bx^{k+1 } &:=&  \prox_f^\bM \big( \bx^k - \bM^{-1}  \bA^\top \bp^k  \big) \\
\bptilde^{k } &:= & \prox_{g^*}^{\bXi}  ( \bp^k + \bXi^{-1}  \bA \bx^{k+1 } ) \\
\bp^{k+1} &: =& \bptilde^k - \bXi^{-1} \bA\bM^{-1} \bA^\top  (\bptilde^k - \bp^k) + \bXi^{-1} \bA ( \bx^{k+1} - \bx^k)    \end{array} \right.  $ \\
 \hline
   \tabincell{l}{ \texttt{PDS-V} \\ \cite[Eq.(30)]{fxue_gopt} }
 & $    \left\lfloor  \begin{array}{lll}
\bptilde^{k } &:= & \prox_{g^*}^\bXi \big(  \bp^k + \bXi^{-1} \bA  \bx^k \big)  \\
\bx^{k+1 } &:= &  \prox_{f }^{\bM}  ( \bx^k - \bM^{-1}  \bA^\top   \bp^{k+1 } ) \\
\bp^{k+1}  &:= &  \bptilde^k -  \bXi^{-1} \bA ( \bx^{k+1}
-\bx^k)   \end{array} \right. $  \\
 \hline
   \tabincell{l}{ \texttt{PDS-VI} \\ \cite[Eq.(31)]{fxue_gopt} }
& $ 
 \left\lfloor  \begin{array}{lll}
\bxtilde^{k } &:= & \prox_f^\bM \big(  \bx^k - \bM^{-1} \bA^\top \bp^k \big)  \\
\bp^{k+1 } &:= &  \prox_{g^* }^{\bXi}  ( \bp^k + \bXi^{-1}  \bA   \bxtilde^{k } ) \\
\bx^{k+1}  &:= &  \bxtilde^k -  \bM^{-1} \bA^\top ( \bp^{k+1}  -\bp^k)    \end{array} \right.   $  \\
 \hline
   \tabincell{l}{ \texttt{PDS-VII} \\ \cite[Eq.(32)]{fxue_gopt} }
 & $   \left\lfloor  \begin{array}{lll}
\bxtilde^{k } &:= & \prox_f^\bM \big(  \bx^k - \bM^{-1}  \bA^\top \bp^k \big) \\
\bptilde^{k } &:= & \prox_{g^*}^{\bXi}  ( \bp^k + 
\bXi^{-1} \bA   \bx^k )  \\
\bx^{k+1}  &:= & \bxtilde^k -  \bM^{-1} \bA^\top (\bptilde^k - \bp^k)   \\
  \bp^{k+1} &:= & \bptilde^k + \bXi^{-1}  \bA (\bxtilde^k - \bx^k)  \end{array}   \right. $  \\
    \Xhline{1.2pt}  
     \end{tabular}  } 
\vskip 0.5em
\end{table}

\begin{table} [h!] \label{table_pds}
\centering
\caption{The PPA reinterpretations of the proposed PDS algorithms }
\vspace{-.5em}
\hspace*{-.2cm}
\resizebox{.97\columnwidth}{!} {
\begin{tabular}{||c||c|c|c|c||} 
    \Xhline{1.2pt}  
schemes & $\bc$ & $\cA$ & $\cQ$ & $\cM$  \\
\hline
\texttt{PDS-I} &
\multirow{7}{*} {  \tabincell{c}{ \\ \\  \\ \\ \\  $\begin{bmatrix}
\bx \\ \bp \end{bmatrix} $ } } & 
\multirow{7}{*}{  \tabincell{c}{ \\ \\   \\ \\  \\  $\begin{bmatrix}
\partial f &   \bA^\top   \\
-\bA & \partial g^* \end{bmatrix}$ }  } &  
$ \begin{bmatrix}   \bM  & -\bA^\top  \\
-\bA  &  \bXi   \end{bmatrix}   $   & 
\multirow{2}{*}{   $\I_{M+N}  $  }   \\ 
 \cline{1-1}   \cline{4-4} 
\texttt{PDS-II} & &  &  
$ \begin{bmatrix}   \bM   &  \bA^\top    \\
 \bA  &  \bXi \end{bmatrix} $   &    \\ 
 \cline{1-1}   \cline{4-5} 
\texttt{PDS-III} & &  &  
$  \begin{bmatrix}
\bM   & \bf 0     \\  \bA   &  \bXi 
\end{bmatrix} $   &   $  \begin{bmatrix}  
 \I_N - \bM^{-1} \bA^\top \bXi^{-1}\bA  & -\bM^{-1} \bA^\top    \\    \bf 0   & \I_{M} \end{bmatrix}  $    \\ 
 \cline{1-1}   \cline{4-5} 
\texttt{PDS-IV} & &  &  
$  \begin{bmatrix}
\bM   & -\bA^\top    \\  \bf 0   &  \bXi 
\end{bmatrix} $   & 
 $  \begin{bmatrix}
\I_N  &  \bf 0    \\    \bXi^{-1} \bA   & 
\I_{M} - \bXi^{-1}   \bA \bM^{-1} \bA^\top 
\end{bmatrix}  $      \\ 
 \cline{1-1}   \cline{4-5} 
\texttt{PDS-V}  & &  &  
$  \begin{bmatrix}
\bM   & \bf 0     \\  \bA   &  \bXi 
\end{bmatrix} $   & 
 $  \begin{bmatrix}  \I_N  & \bf 0    \\  
 \bXi^{-1} \bA   & \I_{M} \end{bmatrix}  $   \\ 
  \cline{1-1}   \cline{4-5} 
\texttt{PDS-VI} & &  &  
$  \begin{bmatrix} \bM   & -\bA^\top    \\
 \bf 0   &  \bXi \end{bmatrix} $   & 
 $  \begin{bmatrix}
\I_N  &  -\bM^{-1} \bA^\top    \\  
  \bf 0   & \I_{M}  \end{bmatrix}  $   \\ 
 \cline{1-1}   \cline{4-5} 
\texttt{PDS-VII}  & &  &  
$ \begin{bmatrix}   \bM   & -\bA^\top    \\
 \bA  &  \bXi \end{bmatrix} $   &  $ \begin{bmatrix}
\I_N  & -\bM^{-1} \bA^\top   \\
\bXi^{-1}   \bA  & \I_M    \end{bmatrix}  $    \\ 
    \Xhline{1.2pt}  
     \end{tabular}  } 
\vskip 0.5em
\end{table}

\begin{table} [h!] \label{table_pds_sg}
\centering
\caption{The corresponding $\cS$ and $\cG$ of the proposed PDS algorithms }
\vspace{-.5em}
\hspace*{-.2cm}
\resizebox{.97\columnwidth}{!} {
\begin{tabular}{||c||c|c||c||} 
    \Xhline{1.2pt}  
schemes & $\cS$ & $\cG$ &  convergence  condition  \\
\hline
\texttt{PDS-I} &   $ \begin{bmatrix}   \bM  & -\bA^\top  \\
-\bA  &  \bXi   \end{bmatrix}   $   & $ \begin{bmatrix}   \bM  & -\bA^\top  \\
-\bA  &  \bXi   \end{bmatrix}   $ & 
\multirow{4}{*}{  \tabincell{c}{ \\ $\bM,\bXi \in \bbS_{++}$ \\ \\ $\bM \succ  \bA^\top \bXi^{-1} \bA $ \\
or $\bXi \succ  \bA \bM^{-1} \bA^\top $  }  }  \\ 
 \cline{1-3}   
\texttt{PDS-II} & $ \begin{bmatrix}   \bM   &  \bA^\top    \\
 \bA  &  \bXi \end{bmatrix} $   & $ \begin{bmatrix}   \bM   &  \bA^\top    \\  \bA  &  \bXi \end{bmatrix} $   &   \\ 
 \cline{1-3}   
\texttt{PDS-III} & $   \begin{bmatrix}
\bM^{-1}  & - \bM^{-1} \bA^\top \bXi^{-1}    \\  
 - \bXi^{-1} \bA \bM^{-1} & \bXi^{-1}   \end{bmatrix}^{-1} $   & 
 $  \begin{bmatrix}
\bM  +  \bA^\top \bXi^{-1} \bA &  \bA^\top     \\  
 \bA & \bXi   \end{bmatrix}  $    &    \\ 
 \cline{1-3}    
\texttt{PDS-IV} & $  \begin{bmatrix}
\bM^{-1}  & \bM^{-1} \bA^\top \bXi^{-1}    \\  
 \bXi^{-1} \bA \bM^{-1} & \bXi^{-1}   \end{bmatrix}^{-1} $   & 
 $   \begin{bmatrix}
\bM  & -\bA^\top     \\  
-\bA & \bXi + \bA \bM^{-1} \bA^\top   \end{bmatrix} $    &     \\ 
 \hline  
\texttt{PDS-V}  & \multirow{3}{*}{ \tabincell{c}{ \\ \\ $   \begin{bmatrix}  \bM  & \bf 0   \\
 \bf 0    & \bXi  \end{bmatrix} $ } }  & 
 $  \begin{bmatrix}
\bM - \bA^\top \bXi^{-1} \bA & \bf 0    \\  
 \bf 0    & \bXi   \end{bmatrix}   $    & 
  \tabincell{c}{      $\bM,\bXi  \in \bbS_{++}$ \\    $ \bM \succ \bA^\top \bXi^{-1} \bA $  }   \\ 
  \cline{1-1} \cline{3-4}
\texttt{PDS-VI} &   & 
 $ \begin{bmatrix} \bM  & \bf 0    \\  
 \bf 0    & \bXi - \bA\bM^{-1} \bA^\top  \end{bmatrix}    $    &  \tabincell{c}{      $\bM,\bXi  \in \bbS_{++}$ \\    $ \bXi \succ \bA \bM^{-1} \bA^\top $  }  \\ 
  \cline{1-1} \cline{3-4}
\texttt{PDS-VII}  &    &  $\begin{bmatrix}
\bM - \bA^\top \bXi^{-1} \bA & \bf 0    \\  
 \bf 0    & \bXi - \bA\bM^{-1} \bA^\top     \end{bmatrix}  $    &   \tabincell{c}{ $\bM,\bXi  \in \bbS_{++}$ \\ $\bM \succ    \bA^\top \bXi^{-1} \bA$ \\
    $ \bXi  \succ  \bA \bM^{-1} \bA^\top   $ }  \\ 
    \Xhline{1.2pt}  
     \end{tabular}  } 
\vskip 0.5em
\end{table}

The connections of the proposed PDS algorithms to the previous works, e.g. \cite[Algorithms 5.1 and 5.2]{condat_2013},  \cite[Theorems 3.1 and 4.2]{plc_2012} and 
 \cite[Algorithm 2.1]{bot_jmiv_2014}, have been discussed in \cite[Sect. 5]{fxue_gopt}.

\subsection{The generalized Bregman distance and ergodic primal-dual gap} 
\label{sec_pds_gap}
Similar to Sect. \ref{sec_lag_bregman} and \ref{sec_lag_gap}, the unified PPA framework also facilitates the gap analysis for the PDS algorithms.  
\begin{lemma} \label{l_gap_pd}
Given the primal-dual form $\cL(\bx,\bp)$ as \eqref{pd}, consider all the PDS  schemes listed in Table 4, where $\bctilde^k = (\bxtilde^k, \bptilde^k) $ denotes the proximal output, when the schemes are interpreted by the PPA (shown in Table 5). Then,  the following holds, $\forall \bc =(\bx, \bp) \in \R^N   \times \R^M$:

{\rm (i)} $\bPi(\bctilde^k, \bc) \le \big  \langle \cQ (\bctilde^k - \bc^k)  \big| \bc -\bctilde^k  \big \rangle$,

{\rm (ii)} $\bPi  \big( \frac{1}{k}  \sum_{i=0}^{k-1} \bctilde^i, \bc  \big)  \le \frac{1}{2k} \big\| \bc^0 - \bc \big\|_\cS^2$.
\end{lemma}

\begin{proof}
(i) First, note that the proximal step of all the PDS  schemes listed in Table 4 can be written as:
\[
\begin{bmatrix}
\bf 0 \\ \bf 0  \end{bmatrix} \in 
\begin{bmatrix}
\partial f &  \bA^\top \\
-\bA  & \partial g^*   \end{bmatrix}
\begin{bmatrix}
\bxtilde^k \\  \bptilde^k \end{bmatrix} 
+ \begin{bmatrix}
\text{---}\bQ_1\text{---} \\ 
\text{---}\bQ_2\text{---}   \end{bmatrix} 
(\bctilde^k - \bc^k ),  
\]
which is
\be \label{q33}
\left\lfloor \begin{array}{llll}
\bf 0  & \in  & \partial f(\bxtilde^k) +\bA^\top \bptilde^k  + \bQ_1 (\bctilde^k - \bc^k), \\ 
\bf 0  & \in  & \partial g^*(\bptilde^k)  -  \bA\bxtilde^k  + \bQ_2 (\bctilde^k - \bc^k). 
 \end{array}  \right.
\ee
Then, by convexity of $f$ and $g^*$, we develop:
\begin{eqnarray}
f(\bx) & \ge & f(\bxtilde^k) + \langle \partial f(\bxtilde^k) | \bx - \bxtilde^k \rangle 
\nonumber \\
& = &  f(\bxtilde^k) - \langle \bA^\top \bptilde^k  | \bx - \bxtilde^k \rangle
- \langle \bQ_1(\bctilde^k - \bc^k)  | \bx - \bxtilde^k \rangle,
\quad \text{by \eqref{q33}} \nonumber
\end{eqnarray}
and
\begin{eqnarray}
g^*(\bp) & \ge & g^*(\bptilde^k) + \langle \partial g^*(\bptilde^k) | \bp - \bptilde^k \rangle 
\nonumber \\
& = &  g^*(\batilde^k) + \langle  \bA \bxtilde^k  | \bp - \bptilde^k \rangle
- \langle \bQ_2(\bctilde^k - \bc^k)  | \bp - \bptilde^k \rangle.
\quad \text{by \eqref{q33}} \nonumber
\end{eqnarray}
Summing up both inequalities yields
\begin{eqnarray}
f(\bx) +g^*(\bp) - f(\bxtilde^k) - g^*(\bptilde^k) 
&\ge & - \langle \bA^\top \bptilde^k  | \bx - \bxtilde^k \rangle
- \langle \bQ_1(\bctilde^k - \bc^k)  | \bx - \bxtilde^k \rangle
\nonumber \\ 
& + & \langle \bA \bxtilde^k  | \bp - \bptilde^k \rangle
- \langle \bQ_2(\bctilde^k - \bc^k)  | \bp - \bptilde^k \rangle.
\nonumber
\end{eqnarray}
Finally, we have
\begin{eqnarray}
&& \cL(\bxtilde^k, \bp) - \cL(\bx, \bptilde^k)
\nonumber \\
& =& f(\bxtilde^k) + g^*(\bptilde^k)  - f(\bx) - g^*(\bp) 
+\langle \bp | \bA\bxtilde^k  \rangle
-\langle \bptilde^k | \bA\bx  \rangle
\nonumber \\
&\le &   \langle \bQ_1(\bctilde^k - \bc^k)  | \bx - \bxtilde^k \rangle +  \langle \bQ_2(\bctilde^k - \bc^k)  | \bp - \bptilde^k \rangle
\nonumber \\ 
& + & \langle \bp | \bA\bxtilde^k  \rangle
-\langle \bptilde^k | \bA\bx  \rangle
+\langle \bA^\top \bptilde^k  | \bx - \bxtilde^k \rangle
-\langle \bA \bxtilde^k  | \bp - \bptilde^k \rangle
\nonumber \\ 
& = & \langle \bQ_1(\bctilde^k - \bc^k)  | \bx - \bxtilde^k \rangle +  \langle \bQ_2(\bctilde^k - \bc^k)  | \bp - \bptilde^k \rangle
\nonumber \\ 
& = & \langle \cQ (\bctilde^k - \bc^k)  | \bc - \bctilde^k \rangle.
\nonumber 
\end{eqnarray}

(ii) similar to the proof of Lemma \ref{l_gap_lag}-(ii).
\end{proof}

Similar to the Lagrangian schemes,  for the PDS schemes,  $\bPi(\bc, \bcstar) $ with  $ \bcstar \in \zer \cA$  essentially belongs to the generalized Bregman distance associated with $q(\bc):=f(\bx) +g^*(\bp)$ between any point $\bc= (\bx,  \bp)$ and a saddle point  $\bcstar= (\bxstar, \bpstar )$, which satisfies $0 \le D_q^\flat(\bc,\bc^\star) \le \bPi(\bc, \bcstar)  \le  D_q^\sharp (\bc,\bc^\star)$. Indeed, the generalized Bregman distance is given as
\begin{eqnarray}
0 & \le & D_q^\flat(\bc,\bcstar) = \inf_{\bv \in \partial q(\bcstar)} q(\bc) - q(\bcstar) + \langle \bv | \bcstar - \bc \rangle 
\nonumber \\
&=&  f(\bx) - f(\bxstar) + \inf_{\bv \in \partial f(\bxstar)} \langle \bv  |  \bxstar - \bx \rangle  + g^*(\bp) - g^*(\bpstar) +\inf_{\bt \in \partial g^*(\bpstar)} \langle \bt | \bpstar - \bp \rangle 
\nonumber \\
& \le &  f(\bx) - f(\bxstar) + \langle \bA^\top \bpstar  | \bx - \bxstar \rangle  + g^*(\bp) - g^*(\bpstar) - \langle   \bA \bxstar  | \bp - \bpstar \rangle 
\nonumber \\
&=&  f(\bx) - f(\bxstar) + g^*(\bp) - g^*(\bpstar)  + \langle  \bA \bx |\bpstar   \rangle  - \langle   \bA \bxstar  | \bp  \rangle 
\nonumber \\
&=&  \cL(\bx, \bpstar) - \cL(\bxstar,  \bp)
\nonumber \\
& =& \bPi(\bc, \bcstar) \le   D_q^\sharp (\bc,\bcstar). 
\nonumber 
 \end{eqnarray}

Then, we obtain the  convergence rate of $\bPi(\bc^k, \bcstar) $ in an ergodic sense. 
\begin{theorem} \label{t_bregman_pds}
For all the PDS  algorithms shown in Table 4, the generalized  Bregman distance generated by $f(\bx)+g^*(\bp)$ between the ergodic point $\frac{1}{k}  \sum_{i=0}^{k-1} \bctilde^i$ and a saddle point $\bcstar \in \zer \cA$ has a rate of $\cO(1/k)$:
\[
0\le \bPi  \bigg( \frac{1}{k}  \sum_{i=0}^{k-1} \bctilde^i, \bcstar \bigg)  \le \frac{1}{2k} \big\| \bc^0 - \bcstar \big\|_\cS^2,
\]
where $\{\bctilde^i\}_{i\in\N}$ and $\cS$ are defined in Lemma \ref{l_gap_pd} and \ref{l_gppa}.
\end{theorem}
The proof is similar to Theorem \ref{t_bregman_lag}.

\vskip.2cm
Likewise, for the class of PDS algorithms, for given sets $B_1 \subset \R^N$ and $B_2 \subset \R^M$, the {\it primal-dual gap} function restricted to $B_1 \times B_2$ is defined as:
\be \label{gap_pd}
\bPsi_{B_1 \times B_2} (\bc) = \sup_{\bp' \in B_2} \cL(\bx,\bp') -  \inf_{\bx'\in B_1 } \cL(\bx', \bp),
\ee 
which has the upper bound:
\begin{corollary} \label{c_gap_pds}
Under the conditions of Theorem \ref{t_bregman_pds}, if the set    $ B_1 \times B_2$ is bounded, the ergodic primal-dual gap defined as \eqref{gap_pd} has the upper bound:
\[
\bPsi_{B_1 \times B_2}   \bigg( \frac{1}{k}  \sum_{i=0}^{k-1} \bctilde^i  \bigg)  \le \frac{1}{2k}
\sup_{\bc \in B_1 \times B_2 } \big\| \bc^0 - \bc \big\|_\cS^2.
\] 
Furthermore, $\bPsi_{B_1 \times B_2} ( \frac{1}{k}  \sum_{i=0}^{k-1} \bctilde^i  ) \ge 0$, if the set $B_1 \times B_2$ contains a saddle point $\bcstar = (\bxstar, \bpstar) \in \zer \cA$.
\end{corollary}
The proof is similar to Corollary \ref{c_gap_lag}.

\begin{remark}
For  \texttt{PDS-I} and \texttt{II} with corresponding $\cM=\cI$, Theorem \ref{t_bregman_pds} and Corollary \ref{c_gap_pds} can be simplified as  $\bPi  \big( \frac{1}{k}  \sum_{i=1}^{k} \bc^i, \bcstar  \big)  \le  \frac{1}{2k} \big\| \bc^0 - \bcstar \big\|_\cQ^2$ and 
$\bPsi_{B_1 \times B_2} $ $ \big( \frac{1}{k}  \sum_{i=1}^{k} \bc^i  \big)  \le \frac{1}{2k}
\sup_{\bc \in B_1 \times B_2 } \big\| \bc^0 - \bc \big\|_\cQ^2$.
\end{remark}

\begin{remark} \label{r_gap_pds}
Similarly to Remark \ref{r_gap_lag}, under additional conditions on $f$ and $g^*$, one can obtain {\it the convergence rate of $\cO(1/k)$ of the sequence of the primal value of \eqref{primal}, evaluated at the ergodic averaging point $\frac{1}{k} \sum_{i=0}^{k-1} \bctilde^i$}, namely, it holds that:
\be \label{gap_pds}
   f\bigg(\frac{1}{k}  \sum_{i=0}^{k-1} \tilde{\bx}^i\bigg) 
+ g\bigg( \bA \Big( \frac{1}{k}  \sum_{i=0}^{k-1} \tilde{\bx}^i \Big) \bigg) -  f(\bxstar) - g(\bA \bxstar)  
 \le \frac{1}{2k}
\sup_{\bc \in B_1 \times B_2} \big\| \bc^0 - \bc \big\|_\cS^2.
\ee

Indeed, if $\dom f$ and $\dom g^*$ are bounded, then we can simply take the sets $B_1 = \dom f$ and $B_2=\dom g^*$.    Denoting the ergodic averaging points by
 $\bchat^k = \frac{1}{k}  \sum_{i=0}^{k-1} \tilde{\bc}^i$
 ($\bxhat^k$ and $\bphat^k$ are defined similarly),
 using Fenchel-Young inequality \cite[Proposition 13.15]{plc_book}, we develop
\begin{eqnarray}
&& \bPsi_{B_1 \times B_2}   \big( \bchat^k  \big)  
\nonumber \\ 
&=& \sup_{\bp' \in B_2} \cL \big(
\bxhat^k ,  \bp' \big) - \inf_{\bx' \in B_1} \cL \big(
\bx', \bphat^k  \big) 
\nonumber \\ 
&=& \sup_{\bp' \in B_2} f\big( \bxhat^k \big) 
+\big\langle \bp' \big| \bA \bxhat^k  \big\rangle - g^*(\bp')
 -   \inf_{ \bx'   \in B_1} \big( f(\bx') +
\big\langle \bphat^k  \big| \bA \bx'  \big\rangle  - g^*\big( 
\bphat^k \big) \big)
\nonumber \\ 
& \ge &  f\big( \bxhat^k  \big) 
+ g\big( \bA \bxhat^k  \big) + g^*\big(  \bphat^k  \big) 
-  f(\bxstar) - \big\langle \bphat^k   \big| \bA \bxstar \big\rangle  
\nonumber \\ 
& \ge &  f\big( \bxhat^k  \big) 
+ g\big( \bA \bxhat^k  \big) -  f(\bxstar) - g(\bA \bxstar), 
\nonumber 
\end{eqnarray} 
which, combining with Corollary \ref{c_gap_pds}, yields \eqref{gap_pds}.
Still, \eqref{gap_pds} holds for all the PDS algorithms shown in Table 4.
\end{remark}

\subsection{Reductions of some PDHG algorithms}
\label{sec_red_pds}
By Corollary \ref{c_gppa}-(iii), \texttt{PDS-I} and \texttt{II} can be reduced to a simple resolvent \eqref{aa}, where $\bv^k = \cQ^{\frac{1}{2} }   \bc^k$, $\cQ$ is specified in Table 5 for \texttt{PDS-I} or \texttt{II}.

\paragraph{A degenerate case} In particular, if $\bA=\I_N$, $\bM=\I_N$, $\bXi=\I_N$, then, $\cQ=\begin{bmatrix}
\I_N & -\I_N \\ -\I_N &  \I_N 
\end{bmatrix}$ for \texttt{PDS-I} or $\cQ=\begin{bmatrix}
\I_N &  \I_N \\   \I_N &  \I_N 
\end{bmatrix}$ for \texttt{PDS-II}, which becomes degenerate. The convergence of this case, which is not covered by Table 6, can be answered by the degenerate analysis.   

As an example, let us consider \texttt{PDS-I}, which becomes
\be \label{er}
\left\lfloor  \begin{array}{l}
\bx^{k+1} := \prox_f \big( \bx^k -   \bp^k \big),  \\
\bp^{k+1} := \prox_{g^*}  \big( \bp^k +  2\bx^{k+1} - \bx^k  \big).
 \end{array}   \right.
\ee
The metric  $\cQ$ can be decomposed as  $\cQ=\cD \cD^\top =\begin{bmatrix}
\I_N \\  -\I_N  \end{bmatrix} \begin{bmatrix}
\I_N &  -\I_N  \end{bmatrix} $. Then, following the procedure similar to Proposition \ref{p_red_lag}, we obtain the reduced PPA as:
\[
\bv^{k+1} = \big( \cI + (\cD^\top \cA^{-1} \cD)^{-1} \big)^{-1} \bv^k, 
\]
where $\bv^k = \cD^\top \bc^k = \bx^k - \bp^k$.

The active variable of \eqref{er}  can also be identified without the degenerate PPA analysis. Indeed, from \eqref{er}, we have
\[
  \bx^{k+1} - \bp^{k+1}  =
   \prox_{f}  \big(   \bx^k -    \bp^{k} \big)
-   \prox_{g^*}  \big( \bp^k+2  \bx^{k+1} -    \bx^{k} \big).
\]
Denoting $\bv^k:=\bx^k-\bp^k$, it becomes
\begin{eqnarray}
  \bv^{k+1}  &  =&   \prox_{f}  \big(   \bv^k \big)
-   \prox_{g^*}  \big(  2  \bx^{k+1} -    \bv^{k} \big)
  \nonumber\\
  &=&  \prox_{f}  \big(   \bv^k \big)
-   \prox_{g^*}  \big(  2  \prox_{f} (\bv^k ) -    \bv^{k} \big)
  \nonumber\\
  &=& \Big( \prox_{f}  - \prox_{g^*} \circ (2  \prox_{f}-\cI)  \Big) (\bv^k)
  \nonumber\\
  &=& \Big( \prox_{f}  - (\cI - \prox_{g} )  \circ (2  \prox_{f}-\cI)  \Big) (\bv^k)
  \nonumber\\
  &=& \Big( \cI - \prox_{f} + \prox_{g}   \circ (2  \prox_{f}-\cI)  \Big) (\bv^k),
  \nonumber
\end{eqnarray}
which shows that \eqref{er} is essentially a DRS algorithm.

\section{Operator splitting based on mixed strategies}
\label{sec_mix}
Consider the hybrid strategy proposed in \cite[Sect. 6]{fxue_gopt}, which aims at minimizing \cite[Eq.(34)]{fxue_gopt}:
\be \label{obj_another}
\min_{\bx}   f(\bx) + g(\bA\bx)  + h (\bB \bx), 
\ee 
where $\bx \in \R^N$, $\bA: \R^N \mapsto \R^{M_1}$, 
$\bB: \R^N \mapsto \R^{M_2}$, 
$f: \R^{N} \mapsto \R\cup \{+\infty\}$, $g: \R^{M_1} \mapsto \R\cup \{+\infty\}$,  $h: \R^{M_2} \mapsto \R\cup \{+\infty\}$.  

\subsection{The hybrid schemes and their PPA interpretations}
\label{sec_mix_algo}
Taking Lagrangian of  $g$, and applying primal-dual to  $h$ in \eqref{obj_another} yields \cite[Eq(35)]{fxue_gopt}: 
\be \label{obj_mix}
\cL(\bx,\ba,\bb,\bp) := f(\bx) + g(\ba) + \bp^\top (\bA\bx - \ba) +   \bb^\top \bB \bx -  h^*(\bb), 
\ee 
or 
\be \label{obj_mix_aug}
\cL_\bTheta (\bx,\ba,\bb,\bp) := f(\bx) + g(\ba) + \bp^\top (\bA\bx - \ba) 
+\frac{1}{2} \big\|\bA - \ba\big\|_{\bTheta}^2 +   \bb^\top \bB \bx -  h^*(\bb), 
\ee 
where $\bp \in \R^{M_1}$, $\ba   \in \R^{M_1}$,  $\bb \in \R^{M_2}$, we devise  the  hybrid schemes based on the alternating optimization of \eqref{obj_mix} or \eqref{obj_mix_aug}, shown in Table 7. \texttt{MIX-I,III,IV} and \texttt{V} can be found in \cite[Sect. 6]{fxue_gopt}, and are extended to general proximal metrics here.

\begin{table} [h!] \label{table_mix_algo}
\centering
\caption{The proposed hybrid algorithms }
\vspace{-.5em}
\hspace*{-.2cm}
\resizebox{.97\columnwidth}{!} {
\begin{tabular}{|l|l|} 
    \Xhline{1.2pt}  
    name  & iterative scheme \\
\hline 
\tabincell{l}{ \texttt{MIX-I} \\ \cite[Eq.(36)]{fxue_gopt} }
 & $ \left\lfloor  \begin{array}{lll}
\bx^{k+1} & = & \prox_f^\bM\big( \bx^k - \bM^{-1} (\bA^\top \bp^k +\bB^\top \bb^k  )  \big) \\
  \ba^{k+1} & =  & \prox_{g}^{\bOmega} ( \ba^{k } + \bOmega^{-1}  \bp^k)    \\ 
  \bb^{k+1} & =  & \prox_{h^*}^{\bXi^{-1}} ( \bb^k + 
  \bXi \bB (2 \bx^{k+1} - \bx^k) )    \\ 
\bp^{k+1}  & = & \bp^k + \bTheta \big(  \bA (2  \bx^{k+1}
-\bx^k)  - (2  \ba^ {k+1} - \ba^k)  \big)  
 \end{array} \right.  $ \\  
 \hline
 \texttt{MIX-II}  & $ \left\lfloor  \begin{array}{lll}
\bx^{k+1} & = & \prox_f^\bM \big(  \bx^k -\bM^{-1} (  \bB^\top \bb^k + \bA^\top \bp^k ) \big)  \\
\bp^{k+1}  & := & \bp^k + \bTheta  \big(  \bA (2 \bx^{k+1}  - \bx^k )  - \ba^{k} \big)  \\ 
  \ba^{k+1} & =  & \prox_{g}^{ \bOmega } \big(
  \ba^k  +   \bOmega^{-1} (2 \bp^{k+1} -  \bp^k)  \big)   \\ 
  \bb^{k+1} & =  & \prox_{h^*}^{ \bXi^{-1}} \big( \bb^k + 
  \bXi \bB (2 \bx^{k+1} - \bx^k)  \big)    \\ 
 \end{array} \right. $ \\
 \hline
\tabincell{l}{ \texttt{MIX-III} \\ \cite[Eq.(39)]{fxue_gopt} }
 & $  \left\lfloor  \begin{array}{lll}
\bx^{k+1} &: = &  \prox_f^{\bM +\bA^\top \bTheta \bA} \big( (\bM+\bA^\top \bTheta\bA)^{-1} (\bM \bx^k +\bA^\top \bTheta \ba^k -   \bB^\top \bb^k - \bA^\top \bp^k ) \big) \\
\bp^{k+1}  & := & \bp^k + \bTheta  \big(  \bA  \bx^{k+1}   - \ba^k \big)  \\ 
  \ba^{k+1} & =  & \prox_{g}^{\bOmega}  ( \ba^k + \bOmega^{-1} (2\bp^{k+1} - \bp^k) )    \\ 
  \bb^{k+1} & =  & \prox_{h^*}^{\bXi^{-1}} (\bb^k + 
  \bXi \bB (2 \bx^{k+1} - \bx^k)  )    \\ 
 \end{array} \right.   $\\
 \hline
 \tabincell{l}{ \texttt{MIX-IV} \\ \cite[Eq.(40)]{fxue_gopt} }
 & $  \left\lfloor  \begin{array}{lll}
\bx^{k+1} & = &  \prox_f^\bM\big( \bx^k - \bM^{-1} (\bA^\top \bp^k +\bB^\top \bb^k  ) \big)  \\
  \ba^{k+1} & =  & \prox_{g}^{\bOmega+\bTheta}  \big( (\bOmega+\bTheta )^{-1} (\bOmega \ba^k +  \bp^{k }
  +  \bTheta \bA\bx^{k+1} ) \big)    \\ 
  \bb^{k+1} & =  & \prox_{h^*}^{\bXi^{-1} } (\bb^k + \bXi \bB \bx^{k +1}  ) +\bXi  \bB (\bx^{k+1} - \bx^k)    \\ 
\bp^{k+1} & = &\bp^k +  \bTheta  (2 \bA 
 \bx^{k+1} - \bA \bx^k   -  \ba^{k+1}  ) \\
 \end{array} \right.   $ \\
 \hline
 \tabincell{l}{ \texttt{MIX-V} \\ \cite[Eq.(37)]{fxue_gopt} }
  & $   \left\lfloor  \begin{array}{lll}
\bx^{k+1} & = &  \prox_f^\bM\big(  \bx^k - \bM^{-1} (\bA^\top \bp^k +\bB^\top \bb^k  ) \big)  \\
  \ba^{k+1} & =  & \prox_{g}^{\bOmega}  (\ba^k + \bOmega^{-1} \bp^{k } )    \\ 
  \bb^{k+1} & =  & \prox_{h^*}^{\bXi^{-1}} ( \bb^k+ \bXi \bB \bx^{k +1})  +\bXi  \bB (\bx^{k+1} - \bx^k)  \\ 
\bp^{k+1}  & = & \bp^k + \bTheta  \big(  \bA (2  \bx^{k+1}  -\bx^k)  - (2  \ba^ {k+1} - \ba^k)  \big)  \\
 \end{array} \right. $  \\
 \hline
 \texttt{MIX-VI}  & $  \left\lfloor  \begin{array}{lll}
\bx^{k+1} &: = &  \prox_f^{\bM +\bA^\top \bTheta \bA} \big( (\bM+\bA^\top \bTheta \bA)^{-1} (\bM \bx^k +\bA^\top \bTheta  \ba^k -   \bB^\top \bb^k - \bA^\top \bp^k ) \big) \\
  \ba^{k+1} & :=  & \prox_g^{\bOmega  } 
   \big( \ba^k + \bOmega^{-1} ( \bTheta \bA \bx^{k+1} - \bTheta  \ba^k   + \bp^k) \big)    \\ 
  \bb^{k+1 } & :=  & \prox_{h^*}^{\bXi^{-1}}  \big( \bb^k + \bXi \bB (2 \bx^{k+1 } - \bx^k  ) \big)   \\ 
\bp^{k+1}  & := & \bp^k + \bTheta  \big(  \bA  \bx^{k+1}  - \ba^{k +1} \big)  \\ 
 \end{array} \right.    $  \\
    \Xhline{1.2pt}  
     \end{tabular}  } 
\vskip 0.5em
\end{table}

\begin{table} [h!] \label{table_mix}
\centering
\caption{The PPA reinterpretations of the hybrid algorithms }
\vspace{-.5em}
\hspace*{-.2cm}
\resizebox{.97\columnwidth}{!} {
\begin{tabular}{||c||c|c|c|c||} 
    \Xhline{1.2pt}  
schemes & $\bc$ & $\cA$ & $\cQ$ & $\cM$  \\
\hline
\texttt{MIX-I} &
\multirow{6}{*} {  \tabincell{c}{ \\  \\  \\ \\ \\ \\ \\ \\ $\begin{bmatrix}
\bx \\ \ba \\ \bb \\  \bp \end{bmatrix} $ } } & 
\multirow{6}{*}{  \tabincell{c}{ \\  \\  \\ \\ \\ \\ \\ \\ 
  $ \begin{bmatrix}
\partial f & \bf 0  & \bB^\top & \bA^\top  \\
\bf 0 & \partial g  & \bf 0 & -\I_{M_1}   \\
-\bB  & \bf 0  & \partial h^*  & \bf 0   \\
-\bA  &  \I_{M_1}  &   \bf 0 &  \bf 0  
\end{bmatrix} $ }  } &  $ \begin{bmatrix}
\bM & \bf 0 & -\bB^\top  & - \bA^\top   \\
\bf 0 & \bOmega    & \bf 0 & \I_{M_1}  \\
-\bB & \bf 0  &   \bXi^{-1}   & \bf 0  \\
 -\bA &  \I_{M_1}   & \bf 0 & \bTheta^{-1}     
\end{bmatrix}  $ & \multirow{3}{*}{  \tabincell{c}{ 
\\  \\  \\ \\   $\I_{N+2M_1+M_2}$ } }  \\ 
 \cline{1-1}    \cline{4-4}  
\texttt{MIX-II}  & &  & 
$  \begin{bmatrix}
\bM  & \bf 0 & -\bB^\top  & -\bA^\top   \\
\bf 0 &  \bOmega  & \bf 0 & -\I_{M_1} \\
-\bB   & \bf 0  &   \bXi^{-1}   & \bf 0  \\
 -\bA &  -\I_{M_1}    & \bf 0 & \bTheta^{-1}
\end{bmatrix}  $  &    \\ 
 \cline{1-1}   \cline{4-4}  
\texttt{MIX-III} & &  & 
$ \begin{bmatrix}
 \bM & \bf 0 &  -\bB^\top & \bf 0   \\
\bf 0 & \bOmega & \bf 0   &  -\I_{M_1}   \\
 -\bB &  \bf 0   & \bXi^{-1}  & \bf 0    \\
\bf 0  & -\I_{M_1}   & \bf  0   &  \bTheta^{-1}  
\end{bmatrix}  $  &     \\ 
 \cline{1-1}   \cline{4-5} 
\texttt{MIX-IV} & &  & 
$  \begin{bmatrix}
\bM  & \bf 0 & -\bB^\top  & - \bA^\top   \\
\bf 0 & \bOmega  & \bf 0 & \bf 0  \\
\bf 0  & \bf 0  &   \bXi^{-1}   & \bf 0  \\
 \bf 0 &  \bf 0    & \bf 0 & \bTheta^{-1}    
\end{bmatrix}  $  & $ \begin{bmatrix}
\I_N  & \bf 0 & \bf 0  &  \bf 0   \\
\bf 0 & \I_{M_1}   & \bf 0 & \bf 0    \\
 \bXi \bB   & \bf 0  &   \I_{M_2}    & \bf 0  \\
 \bTheta \bA  &  \bf 0    & \bf 0 & \I_{M_1}   
\end{bmatrix} $    \\ 
 \cline{1-1}   \cline{4-5} 
\texttt{MIX-V} & &  & $  \begin{bmatrix}
\bM  & \bf 0 & -\bB^\top  & - \bA^\top   \\
\bf 0 & \bOmega  & \bf 0 & \I_{M_1}  \\
\bf 0  & \bf 0  &   \bXi^{-1}   & \bf 0  \\
 \bf 0 &  \bf 0    & \bf 0 & \bTheta^{-1}    
\end{bmatrix}  $    & 
 $ \begin{bmatrix}
\I_N  & \bf 0 & \bf 0 & \bf 0  \\
\bf 0 &   \I_{M_1}  & \bf 0  & \bf 0 \\
\bXi \bB   & \bf 0 &   \I_{M_2}  & \bf 0 \\
\bTheta \bA    & - \bTheta & \bf 0 &  \I_{M_1} 
\end{bmatrix}  $       \\ 
 \cline{1-1}   \cline{4-5} 
\texttt{MIX-VI}  & &  & $\begin{bmatrix}
 \bM  & \bf 0 & -\bB^\top  & \bf 0 \\
\bf 0 &   \bOmega  & \bf 0 & \bf 0  \\
-\bB   & \bf 0  &   \bXi^{-1}   & \bf 0  \\
 \bf 0 &  -\I_{M_1}   & \bf 0 & \bTheta^{-1}     
\end{bmatrix}  $   & 
 $ \begin{bmatrix}
\I_N  & \bf 0 & \bf 0 & \bf 0  \\
\bf 0 &   \I_{M_1}  & \bf 0  & \bf 0 \\
\bf 0   & \bf 0 &   \I_{M_2}  & \bf 0 \\
\bf 0   & - \bTheta  & \bf 0 &  \I_{M_1} 
\end{bmatrix}  $   \\
    \Xhline{1.2pt}  
     \end{tabular}  } 
\vskip 0.5em
\end{table}

\begin{table} [h!] \label{table_mix_sg}
\centering
\caption{The corresponding $\cS$ and $\cG$ of the hybrid algorithms }
\vspace{-.5em}
\hspace*{-.2cm}
\resizebox{.97\columnwidth}{!} {
\begin{tabular}{||c||c|c||c||} 
    \Xhline{1.2pt}  
schemes & $\cS$ & $\cG$ &  convergence condition \\
\hline
\texttt{MIX-I} &  $ \begin{bmatrix}
\bM & \bf 0 & -\bB^\top  & - \bA^\top   \\
\bf 0 & \bOmega    & \bf 0 & \I_{M_1}  \\
-\bB & \bf 0  &   \bXi^{-1}   & \bf 0  \\
 -\bA &  \I_{M_1}   & \bf 0 & \bTheta^{-1}     
\end{bmatrix}  $ &    $ \begin{bmatrix}
\bM & \bf 0 & -\bB^\top  & - \bA^\top   \\
\bf 0 & \bOmega    & \bf 0 & \I_{M_1}  \\
-\bB & \bf 0  &   \bXi^{-1}   & \bf 0  \\
 -\bA &  \I_{M_1}   & \bf 0 & \bTheta^{-1}     
\end{bmatrix}  $   & 
\multirow{2}{*}{ \tabincell{c}{ \\ \\   $\bM, \bOmega,\bTheta,\bXi \in \bbS_{++}$ \\ $\bM \succ   \bA^\top \bTheta \bA + \bB^\top \bXi \bB $ \\  $\bOmega  \succ  \bTheta$ } }  \\ 
 \cline{1-3}    
\texttt{MIX-II}  & $  \begin{bmatrix}
\bM  & \bf 0 & -\bB^\top  & -\bA^\top   \\
\bf 0 &  \bOmega  & \bf 0 & -\I_{M_1} \\
-\bB   & \bf 0  &   \bXi^{-1}   & \bf 0  \\
 -\bA &  -\I_{M_1}    & \bf 0 & \bTheta^{-1}
\end{bmatrix}  $  & $  \begin{bmatrix}
\bM  & \bf 0 & -\bB^\top  & -\bA^\top   \\
\bf 0 &  \bOmega  & \bf 0 & -\I_{M_1} \\
-\bB   & \bf 0  &   \bXi^{-1}   & \bf 0  \\
 -\bA &  -\I_{M_1}    & \bf 0 & \bTheta^{-1}
\end{bmatrix}  $   &  \\ 
 \hline 
\texttt{MIX-III} & $ \begin{bmatrix}
 \bM & \bf 0 &  -\bB^\top & \bf 0   \\
\bf 0 & \bOmega & \bf 0   &  -\I_{M_1}   \\
 -\bB &  \bf 0   & \bXi^{-1}  & \bf 0    \\
\bf 0  & -\I_{M_1}   & \bf  0   &  \bTheta^{-1}  
\end{bmatrix}  $  & $ \begin{bmatrix}
 \bM & \bf 0 &  -\bB^\top & \bf 0   \\
\bf 0 & \bOmega & \bf 0   &  -\I_{M_1}   \\
 -\bB &  \bf 0   & \bXi^{-1}  & \bf 0    \\
\bf 0  & -\I_{M_1}   & \bf  0   &  \bTheta^{-1}  
\end{bmatrix}  $   &
\tabincell{c}{ $\bM, \bOmega, \bXi,\bTheta \in \bbS_{++}$ \\   $\bM \succ  \bB^\top \bXi \bB$ \\ 
$\bOmega \succ \bTheta$ }  \\ 
 \hline  
\texttt{MIX-IV} & $  \begin{bmatrix}
\bM +\bA^\top \bTheta \bA +\bB^\top \bXi \bB  & \bf 0 & -\bB^\top &  -\bA^\top   \\
\bf 0 & \bOmega   & \bf 0 & \bf 0 \\
-\bB  & \bf 0  &   \bXi^{-1}   & \bf 0  \\
-\bA  &   \bf 0   & \bf 0 & \bTheta^{-1}
\end{bmatrix}   $  & $ \begin{bmatrix}
\bM   & \bf 0 & - \bB^\top   &  -\bA^\top  \\
\bf 0 & \bOmega   & \bf 0 & \bf 0  \\
-\bB  & \bf 0  &   \bXi^{-1}     &  \bf 0  \\
 - \bA  &  \bf 0   &  \bf 0  & \bTheta^{-1} 
\end{bmatrix} $  &
\tabincell{c}{ $\bM, \bTheta, \bXi \in \bbS_{++}$ \\ 
$\bOmega \in \bbS_+$ \\ $\bM \succ  \bA^\top \bTheta \bA+  \bB^\top \bXi \bB $ }  \\ 
 \hline  
\texttt{MIX-V} & $  \begin{bmatrix}
\bM +\bA^\top \bTheta \bA +\bB^\top \bXi \bB 
 & -\bA^\top \bTheta & -\bB^\top &  -\bA^\top \\
-\bTheta \bA  & \bOmega+\bTheta   & \bf 0 & \I_{M_1} \\
-  \bB    & \bf 0  &   \bXi^{-1}    & \bf 0   \\
-\bA    &  \I_{M_1} & \bf 0  & \bTheta^{-1} 
\end{bmatrix} $    & 
 $ \begin{bmatrix}
\bM   & \bf 0 & \bB^\top   &  \bA^\top  \\
\bf 0 & \bOmega   & \bf 0 & \I_{M_1} \\
\bB   & \bf 0  &   \bXi^{-1}    & \bf 0  \\
 \bA &  \I_{M_1}  & \bf 0  & \bTheta^{-1} 
\end{bmatrix}  $    &  \tabincell{c}{ $\bM, \bOmega,\bTheta,\bXi \in \bbS_{++}$ \\ $\bM \succ   \bA^\top \bTheta  \bA + \bB^\top \bXi \bB $ \\  $\bOmega  \succ  \bTheta $ }    \\ 
 \hline  
\texttt{MIX-VI}  &  $  \begin{bmatrix}
 \bM  & \bf 0 & -\bB^\top  & \bf 0 \\
\bf 0 &   \bOmega  & \bf 0 & \bf 0  \\
-\bB   & \bf 0  &   \bXi^{-1}   & \bf 0  \\
 \bf 0 &  \bf 0  & \bf 0 & \bTheta^{-1}     
\end{bmatrix}    $   & 
 $ \begin{bmatrix}
 \bM  & \bf 0 & -\bB^\top  & \bf 0 \\
\bf 0 &   \bOmega - \bTheta  & \bf 0 & \bf 0 \\
-\bB   & \bf 0  &   \bXi^{-1}   & \bf 0  \\
 \bf 0 &  \bf 0  & \bf 0 & \bTheta^{-1}   
\end{bmatrix}   $    &  \tabincell{c}{ $\bM, \bOmega,\bTheta,\bXi \in \bbS_{++}$ \\ $\bM \succ   \bB^\top \bXi \bB $ \\  $\bOmega  \succ  \bTheta $ }    \\ 
    \Xhline{1.2pt}  
     \end{tabular}  } 
\vskip 0.5em
\end{table}

The hybrid schemes can be interpreted by alternating optimization. For instance, the $\bx$-updates of \texttt{MIX-I}, \texttt{MIX-II}, \texttt{MIX-IV} and \texttt{MIX-V} are  from non-augmented  \eqref{obj_mix}:
\[
\bx^{k+1}  :=  \arg \min_\bx \cL(\bx,\ba^k,\bb^k,\bp^k) + \frac{1}{2} \|\bx - \bx^k\|_{\bM}^2.
\]
The $\ba$-updates of \texttt{MIX-I}, \texttt{MIX-II}, \texttt{MIX-III} and \texttt{MIX-V} are  from non-augmented  \eqref{obj_mix}. The $\ba$-update of \texttt{MIX-I}, for instance, is
\[
\ba^{k+1}  :=  \arg \min_\ba \cL(\bx^k,\ba,\bb^k,\bp^k) + \frac{1}{2} \|\ba - \ba^k\|_{\bOmega}^2.
\]

The $\bx$-updates of \texttt{MIX-III} and  \texttt{MIX-VI} are  from augmented  \eqref{obj_mix_aug}:
\[
\bx^{k+1}  :=  \arg \min_\bx \cL_\bTheta(\bx,\ba^k,\bb^k,\bp^k) + \frac{1}{2} \|\bx - \bx^k\|_{\bM}^2.
\]
The $\ba$-update of \texttt{MIX-VI} is from
\[
\ba^{k+1}  :=  \arg \min_\ba \cL_\bTheta(\bx^k,\ba,\bb^k,\bp^k) + \frac{1}{2} \|\ba - \ba^k\|_{\bOmega}^2.
\]

Following  the discussion of \cite[Sect. 6]{fxue_gopt}, it is easy to verify that \texttt{MIX-I} and \texttt{III} can be reduced to \texttt{PDS-I} under a certain conditions.

Again, the preconditioning can be applied to the $\bx$-updates of \texttt{MIX-III,IV} or $\ba$-update of \texttt{MIX-IV}.

\subsection{The generalized Bregman distance and ergodic primal-dual gap}
\label{sec_mix_gap}
Similar to Sect. \ref{sec_lag_gap} and \ref{sec_pds_gap}, the  PPA framework could  also provide a unified treatment of the primal-dual gap of the hybrid algorithms.  
\begin{lemma} \label{l_gap_mix}
Given the hybrid  form $\cL(\bx,\ba,\bb,\bp)$ as \eqref{obj_mix}, consider all the hybrid  schemes listed in Table 7, where $\bctilde^k = (\bxtilde^k, \batilde^k, \bbtilde^k, \bptilde^k) $ denotes the proximal output, when the schemes are interpreted by the PPA (shown in Table 8). Then,  the following holds, $\forall \bc =(\bx, \ba, \bb, \bp) \in \R^N   \times \R^{M_1} \times \R^{M_2} \times \R^{M_1}$:

{\rm (i)} $\bPi(\bctilde^k, \bc) \le \big  \langle \cQ (\bctilde^k - \bc^k)  \big| \bc -\bctilde^k  \big \rangle$,

{\rm (ii)} $\bPi  \big( \frac{1}{k}  \sum_{i=0}^{k-1} \bctilde^i, \bc  \big)  \le \frac{1}{2k} \big\| \bc^0 - \bc \big\|_\cS^2$.
\end{lemma}

\begin{proof}
(i) First, note that the proximal step of all the hybrid schemes listed in Table 7 can be written as:
\[
\begin{bmatrix}
\bf 0 \\ \bf 0 \\ \bf 0 \\ \bf 0  \end{bmatrix} \in 
\begin{bmatrix}
\partial f & \bf 0 & \bB^\top & \bA^\top \\
\bf 0 & \partial g & \bf 0  &  -\I_{M_1} \\
-\bB & \bf 0 & \partial h^* & \bf 0 \\ 
-\bA & \I_{M_1} & \bf 0  & \bf 0   \end{bmatrix}
\begin{bmatrix}
\bxtilde^k \\ \batilde^k \\ \bbtilde^k \\ \bptilde^k \end{bmatrix} 
+ \begin{bmatrix}
\text{---}\bQ_1\text{---} \\ 
\text{---}\bQ_2\text{---} \\ 
\text{---}\bQ_3\text{---} \\
\text{---}\bQ_4\text{---} \end{bmatrix} 
(\bctilde^k - \bc^k ),
\]
which is
\be \label{q6}
\left\lfloor \begin{array}{llll}
\bf 0  & \in  & \partial f(\bxtilde^k) +\bB^\top \bbtilde^k +\bA^\top \bptilde^k  + \bQ_1 (\bctilde^k - \bc^k), \\ 
\bf 0  & \in  & \partial g(\batilde^k)  - \bptilde^k  + \bQ_2 (\bctilde^k - \bc^k), \\ 
\bf 0  & \in  & \partial h^*(\bbtilde^k)  - \bB \bxtilde^k  + \bQ_3 (\bctilde^k - \bc^k), \\ 
\bf 0  & =  & -\bA\bxtilde^k +\batilde^k +  \bQ_4 (\bctilde^k - \bc^k). 
 \end{array}  \right.
\ee

Then, by convexity of $f$, $g$ and $h^*$, we develop:
\begin{eqnarray}
f(\bx) & \ge & f(\bxtilde^k) + \langle \partial f(\bxtilde^k) | \bx - \bxtilde^k \rangle 
\nonumber \\
& = &  f(\bxtilde^k) - \langle \bB^\top \bbtilde^k  | \bx - \bxtilde^k \rangle- \langle \bA^\top \bptilde^k  | \bx - \bxtilde^k \rangle
- \langle \bQ_1(\bctilde^k - \bc^k)  | \bx - \bxtilde^k \rangle,
\quad \text{by \eqref{q6}} \nonumber
\end{eqnarray}
\begin{eqnarray}
g(\ba) & \ge & g(\batilde^k) + \langle \partial g(\batilde^k) | \ba - \batilde^k \rangle 
\nonumber \\
& = &  g(\batilde^k) + \langle  \bptilde^k  | \ba - \batilde^k \rangle
- \langle \bQ_2(\bctilde^k - \bc^k)  | \ba - \batilde^k \rangle,
\quad \text{by \eqref{q6}} \nonumber
\end{eqnarray}
\begin{eqnarray}
h^*(\bb) & \ge & h^*(\bbtilde^k) + \langle \partial h^*(\bbtilde^k) | \bb - \bbtilde^k \rangle 
\nonumber \\
& = & h^*(\bbtilde^k) + \langle \bB \bxtilde^k  | \bb - \bbtilde^k \rangle
- \langle \bQ_3(\bctilde^k - \bc^k)  | \bb - \bbtilde^k \rangle.
\quad \text{by \eqref{q6}} \nonumber
\end{eqnarray}
Summing up the above three inequalities yields
\begin{eqnarray}
&& f(\bx) +g(\ba) - f(\bxtilde^k) - g(\batilde^k) 
+h^*(\bb) - h^*(\bbtilde^k)
\nonumber \\
&\ge & - \langle \bB^\top \bbtilde^k  | \bx - \bxtilde^k \rangle  - \langle \bA^\top \bptilde^k  | \bx - \bxtilde^k \rangle
- \langle \bQ_1(\bctilde^k - \bc^k)  | \bx - \bxtilde^k \rangle
\nonumber \\ 
& + & \langle \bptilde^k  | \ba - \batilde^k \rangle
- \langle \bQ_2(\bctilde^k - \bc^k)  | \ba - \batilde^k \rangle
\nonumber \\ 
& + & \langle \bB \bxtilde^k  | \bb - \bbtilde^k \rangle
- \langle \bQ_3(\bctilde^k - \bc^k)  | \bb - \bbtilde^k \rangle.
\nonumber
\end{eqnarray}
Finally, we have
\begin{eqnarray}
&& \cL(\bxtilde^k,\batilde^k,\bb,\bp) - \cL(\bx,\ba,\bbtilde^k,\bptilde^k)
\nonumber \\
& =& f(\bxtilde^k) + g(\batilde^k)+ h^*(\bbtilde^k)  - f(\bx) - g(\ba) -h^*(\bb)
\nonumber \\
&+& \langle \bp | \bA\bxtilde^k - \batilde^k \rangle
-\langle \bptilde^k | \bA\bx - \ba \rangle
 -  \langle \bbtilde^k | \bB\bx \rangle
 +  \langle \bb | \bB\bxtilde^k \rangle
\nonumber \\
&\le &   \langle \bQ_1(\bctilde^k - \bc^k)  | \bx - \bxtilde^k \rangle +  \langle \bQ_2(\bctilde^k - \bc^k)  | \ba - \batilde^k \rangle +  \langle \bQ_3(\bctilde^k - \bc^k)  | \bb - \bbtilde^k \rangle
\nonumber \\ 
& + &  \langle \bB^\top \bbtilde^k  | \bx - \bxtilde^k \rangle  + \langle \bA^\top \bptilde^k  | \bx - \bxtilde^k \rangle
  -  \langle \bptilde^k  | \ba - \batilde^k \rangle
- \langle \bB \bxtilde^k  | \bb - \bbtilde^k \rangle
\nonumber \\
&+& \langle \bp | \bA\bxtilde^k - \batilde^k \rangle
-\langle \bptilde^k | \bA\bx - \ba \rangle
 -  \langle \bbtilde^k | \bB\bx \rangle
 +  \langle \bb | \bB\bxtilde^k \rangle
 \nonumber \\
&= &   \langle \bQ_1(\bctilde^k - \bc^k)  | \bx - \bxtilde^k \rangle +  \langle \bQ_2(\bctilde^k - \bc^k)  | \ba - \batilde^k \rangle +  \langle \bQ_3(\bctilde^k - \bc^k)  | \bb - \bbtilde^k \rangle
\nonumber \\ 
&+& \langle   \bA \bxtilde^k -\batilde^k | 
\bp - \bptilde^k \rangle 
 \nonumber \\
&= &   \langle \bQ_1(\bctilde^k - \bc^k)  | \bx - \bxtilde^k \rangle +  \langle \bQ_2(\bctilde^k - \bc^k)  | \ba - \batilde^k \rangle +  \langle \bQ_3(\bctilde^k - \bc^k)  | \bb - \bbtilde^k \rangle
\nonumber \\ 
&+&   \langle \bQ_4(\bctilde^k - \bc^k)  | \bp - \bptilde^k \rangle \quad \text{by  \eqref{q6}}
\nonumber \\ 
& = &   \langle \cQ(\bctilde^k - \bc^k)  | \bc - \bctilde^k \rangle. 
\nonumber 
\end{eqnarray}

(ii) similar to Lemma \ref{l_gap_lag}-(ii) and Lemma \ref{l_gap_pd}-(ii).
\end{proof}

Similar to the Lagrangian and PDS schemes, for the hybrid algorithms, $\bPi(\bc, \bcstar) $ with  $ \bcstar \in \zer \cA$  is essentially a special instance of generalized Bregman distance generated by $q(\bu):=f(\bx) +g(\ba)+h^*(\bb)$ between any point $\bu= (\bx,  \ba,\bb)$ and a saddle point  $\bustar= (\bxstar, \bastar,\bbstar)$.  More specifically, $0 \le D_q^\flat(\bu,\bu^\star) \le \bPi(\bc, \bcstar)  \le  D_q^\sharp (\bu,\bu^\star)$.  Indeed, the generalized Bregman distance is given as
\begin{eqnarray}
0 & \le & D_q^\flat(\bu,\bustar) = \inf_{\bv \in \partial q(\bustar)} q(\bu) - q(\bustar) + \langle \bv | \bustar - \bu \rangle 
\nonumber \\
&=&  f(\bx) - f(\bxstar) + \inf_{\bv \in \partial f(\bxstar)} \langle \bv | \bxstar - \bx \rangle  + g(\ba) - g(\bastar) + \inf_{\bt \in \partial g(\bastar)} \langle \bt | \bastar - \ba \rangle 
\nonumber \\
&+&  h^*(\bb) - h^*(\bbstar) + \inf_{\bs \in \partial h^*(\bbstar)}  \langle \bs | \bbstar - \bb \rangle  
\nonumber \\
& \le &  f(\bx) - f(\bxstar) + \langle \bB^\top \bbstar +\bA^\top \bpstar  | \bx - \bxstar \rangle  + g(\ba) - g(\bastar) - \langle   \bpstar  | \ba - \bastar \rangle 
\nonumber \\
&+&  h^*(\bb) - h^*(\bbstar) - \langle \bB \bxstar  | \bb - \bbstar \rangle  
\nonumber \\
&=&  f(\bx) - f(\bxstar)  + g(\ba) - g(\bastar) + h^*(\bb) - h^*(\bbstar) 
\nonumber \\
&+&   \langle \bB \bx |  \bbstar \rangle - \langle \bB \bxstar   | \bb  \rangle + \langle \bpstar  | \bA\bx  - \ba \rangle  
\nonumber \\
&=&  f(\bx) - f(\bxstar)  + g(\ba) - g(\bastar) + h^*(\bb) - h^*(\bbstar) 
\nonumber \\
&+&   \langle \bB \bx |  \bbstar \rangle - \langle \bB \bxstar   | \bb  \rangle + \langle \bpstar  | \bA\bx  - \ba \rangle  
- \langle \bpstar  | \underbrace{ \bA\bxstar  - \bastar}_{= \bf 0} \rangle  
\nonumber \\
&=&  \cL(\bx,\ba,\bbstar,\bpstar) - 
\cL(\bxstar, \bastar, \bb, \bp)
\nonumber \\
& =& \bPi(\bc, \bcstar)  \le D_q^\sharp (\bu,\bustar). 
\nonumber 
 \end{eqnarray}

Then, we obtain the  convergence rate of the generalized Bregman distance $\bPi(\bc^k, \bcstar) $ in an ergodic sense. 
\begin{theorem} \label{t_bregman_mix}
For all the hybrid  algorithms shown in Table 7, the generalized Bregman distance generated by $f(\bx)+g(\ba)+h^*(\bb)$ between the ergodic point $\frac{1}{k}  \sum_{i=0}^{k-1} \bctilde^i$ and a saddle point $\bcstar \in \zer \cA$ has a rate of $\cO(1/k)$:
\[
0\le \bPi  \bigg( \frac{1}{k}  \sum_{i=0}^{k-1} \bctilde^i, \bcstar \bigg)  \le \frac{1}{2k} \big\| \bc^0 - \bcstar \big\|_\cS^2,
\]
where $\{\bctilde^i\}_{i\in\N}$ and $\cS$ are defined in Lemma \ref{l_gap_mix} and \ref{l_gppa}.
\end{theorem}
The proof is similar to Theorem \ref{t_bregman_lag} or  \ref{t_bregman_pds}.

\vskip.2cm
Likewise, for the class of hybrid algorithms, for given sets $B_1 \subset \R^N$, $B_2 \subset \R^{M_1}$, $B_3 \subset \R^{M_2}$ and $B_4 \subset \R^{M_1}$, the {\it primal-dual gap} function restricted to $B_1 \times B_2 \times B_3 \times B_4$ is defined as
\be  \label{gap_mix}
\bPsi_{B_1 \times B_2 \times B_3 \times B_4} (\bc) = \sup_{(\bb', \bp') \in B_3 \times B_4} \cL(\bx,\ba,\bb',\bp') -  \inf_{(\bx',\ba') \in B_1\times B_2 } \cL(\bx',\ba',\bb, \bp),
\ee 
which has the upper bound:
\begin{corollary} \label{c_gap_mix}
Under the conditions of Theorem \ref{t_bregman_mix}, if the set    $ B_1 \times B_2 \times B_3 \times B_4$ is bounded, the primal-dual gap defined as \eqref{gap_mix} has the upper bound:
\[
\bPsi_{B_1 \times B_2 \times B_3 \times B_4}   \bigg( \frac{1}{k}  \sum_{i=0}^{k-1} \bctilde^i  \bigg)  \le \frac{1}{2k}
\sup_{\bc \in B_1 \times B_2 \times B_3 \times B_4} 
\big\| \bc^0 - \bc \big\|_\cS^2.
\] 
Furthermore, $\bPsi_{B_1 \times B_2 \times B_3 \times B_4} ( \frac{1}{k}  \sum_{i=0}^{k-1} \bctilde^i  ) \ge 0$, if the set $B_1 \times B_2 \times B_3 \times B_4$ contains a saddle point $\bcstar = (\bxstar, \bastar, \bbstar, \bpstar) \in \zer \cA$.
\end{corollary}
The proof is similar to Corollary \ref{c_gap_lag} or \ref{c_gap_pds}.

\begin{remark}
For  \texttt{MIX-I, II} and \texttt{III} with corresponding $\cM=\cI$, Theorem \ref{t_bregman_mix} and Corollary \ref{c_gap_mix} can be simplified as  $\bPi  \big( \frac{1}{k}  \sum_{i=1}^{k} \bc^i, \bcstar  \big)  \le  \frac{1}{2k} \big\| \bc^0 - \bcstar \big\|_\cQ^2$ and 
$\bPsi_{B_1 \times B_2 \times B_3 \times B_4} $ $ \big( \frac{1}{k}  \sum_{i=1}^{k} \bc^i  \big)  \le \frac{1}{2k}
\sup_{\bc \in B_1 \times B_2 } \big\| \bc^0 - \bc \big\|_\cQ^2$.
\end{remark}

\begin{remark} \label{r_gap_mix}
Similarly to Remarks \ref{r_gap_lag} and \ref{r_gap_pds}, under additional conditions on $f$, $g$ and $h^*$, one can obtain {\it the convergence rate of $\cO(1/k)$ of the sequence of $f(\bx)+g(\ba) +h(\bB\bx)$, evaluated at the ergodic averaging point $\frac{1}{k} \sum_{i=0}^{k-1} \bctilde^i$}, namely, it holds that:
\be   \label{gap_mix}
f\bigg( \frac{1}{k}  \sum_{i=0}^{k-1} \tilde{\bx}^i\bigg)
+  g\bigg(\frac{1}{k}  \sum_{i=0}^{k-1} \tilde{\ba}^i\bigg) 
+ h\bigg( \bB \Big( \frac{1}{k}  \sum_{i=0}^{k-1} \tilde{\bx}^i \Big) \bigg) -  f(\bxstar) -g(\bastar) - h(\bB \bxstar)  
 \le   C/k,
\ee
for some constant $C$.

To show this, we first rewrite $\cL(\bx,\ba,\bb,\bp) $ in \eqref{obj_mix} as
\[
\cL(\bu,\bv) := q(\bu)  + \langle \bv | \bU \bu \rangle  -  l^*(\bv), 
\] 
where $\bu = \begin{bmatrix}
\bx \\ \ba \end{bmatrix}$, $\bv = \begin{bmatrix}
\bb \\ \bp \end{bmatrix}$, $\bU = \begin{bmatrix}
 \bB & \bf 0 \\  \bA & -\I  \end{bmatrix}$, $q: (\bx,\ba) \mapsto f(\bx) +g(\ba)$, $l: (\bb,\bp) \mapsto h(\bb)$. Since the sequence $\{\tilde{\bp}^k\}_{k\in \N}$ converges by Theorem \ref{t_gppa}, and thus lies in an unknown bounded set  $B_4 \subset \R^{M_1}$. Obviously, $ \frac{1}{k}  \sum_{i=0}^{k-1} \tilde{\bp}^i  \in B_4$, $\bpstar \in B_4$. On the other hand, if $\dom f$, $\dom g$ and $\dom h^*$ are bounded, then we can simply take the sets $B_1 = \dom f$, $B_2=\dom g$ and $B_3=\dom h^*$.    Denoting the ergodic averaging point by
 $\bchat^k = \frac{1}{k}  \sum_{i=0}^{k-1} \tilde{\bc}^i$
($\bxhat^k$, $\bahat^k$, $\bbhat^k$ and $\bphat^k$ are defined similarly),  using Fenchel-Young inequality \cite[Proposition 13.15]{plc_book}, we develop
\begin{eqnarray}
&& \bPsi_{B_1 \times B_2 \times B_3 \times B_4}   \big( 
\bchat^k   \big)  
\nonumber \\ 
&=& \sup_{\bv' \in B_3\times B_4} \cL \big( \buhat^k,  \bv' \big) -   \inf_{\bu' \in B_1 \times B_2} \cL \big(
\bu',  \bvhat^k \big) 
\nonumber \\ 
&=& \sup_{\bv' \in B_3 \times B_4} q\big(\buhat^k \big) 
+\big\langle \bv' \big| \bU \buhat^k  \big\rangle - l^*(\bv')
-  \inf_{ \bu'   \in B_1 \times B_2} \big( q(\bu') +
\big\langle \bvhat^k  \big| \bU \bu'  \big\rangle  - l^*\big( 
\bvhat^k  \big) \big)
\nonumber \\
&=&  q\big( \buhat^k  \big) 
+ l\big( \bU \buhat^k  \big) + l^*\big( \bvhat^k  \big) 
-  \inf_{ \bu'   \in B_1 \times B_2} \big( q(\bu') +
\big\langle \bvhat^k \big| \bU \bu'  \big\rangle  \big)
\nonumber \\ 
& \ge & q\big( \buhat^k  \big) 
+ l\big( \bU \buhat^k  \big) + l^*\big( \bvhat^k  \big) 
- q(\bustar) - 
\big\langle \bvhat^k \big| \bU \bustar  \big\rangle 
\nonumber \\ 
& \ge &  q\big( \buhat^k  \big) 
+ l\big( \bU \buhat^k  \big) - q(\bustar) - l(\bU \bustar)
\nonumber \\ 
& = &  q\big( \buhat^k  \big) 
+ h\big( \bB \bxhat^k  \big) - q(\bustar) - h(\bB \bxstar),
\nonumber 
\end{eqnarray} 
which, combining with Corollary \ref{c_gap_mix}, yields \eqref{gap_mix}.
Still, \eqref{gap_mix} holds for all the hybrid algorithms listed in Table 7. Similarly with Remark \ref{r_gap_lag}, the constant $C$ cannot be easily estimated, since the bounded set $B_4$ is not {\it a priori} known.
\end{remark}

\subsection{Reductions of some hybrid algorithms}
As reported in Table 8, \texttt{MIX-I,II} and \texttt{III} correspond to a standard PPA with $\cM=\cI$, and thus, they can be readily expressed as a resolvent \eqref{aa}, by Corollary \ref{c_gppa}-(iii).

Observing that neither of $\bM$, $\bOmega$, $\bXi$ and $\bTheta$ is allowed to be $\bf 0$ for convergence. It is impossible to reduce any variables (i.e. all the variables are active) in   \texttt{MIX-I, II} and \texttt{III}.

\section{Concluding remarks}
\label{sec_conclusion}
The numerical performance of these splitting algorithms has been reported in \cite[Sect. 7]{fxue_gopt}, which is not discussed here.

The proximal point analysis is shown to be able to (i) provide a unified treatment of the generalized Bregman distance and ergodic primal-dual gap; (ii) identify the active variables and reduce the algorithmic dimensionality. The degeneracy reduction in this paper is essentially an application of the degenerate analysis of \cite{bredies_preprint} to the operator splitting algorithms. An important implication of the degeneracy reduction  is that it is possible to loosen  the strict convergence results (e.g., Theorem \ref{t_gppa} and Corollary \ref{t_gppa}) to positive {\it semi-}definite metric \cite{fxue_1} under a certain conditions, which needs further careful study.

Despite of the success of interpretations using the proximal point analysis demonstrated in \cite{fxue_gopt} and this paper, an evident  limitation is that it cannot deal with, for example, the case of \cite[Sect. 5]{cp_2011}, where only one function is assumed to be strongly convex. It may need to exploit the inner structure of $\cA$ and $\cQ$ based on a subspace analysis, e.g., partially strongly convex operator \cite{partial}, which will further enrich the degenerate theory poineered in \cite{bredies_preprint}. 

\begin{acknowledgements}
I am gratefully indebted to the anonymous reviewers and the editor for helpful discussions, particularly related to the convergence analysis of PPA (Lemma \ref{l_gppa}), the notion of infimal postcomposition (Lemma \ref{l_duality}), the generalized Bregman distance and primal-dual gap (Sect. \ref{sec_lag_bregman}, \ref{sec_lag_gap}, \ref{sec_pds_gap} and \ref{sec_mix_gap}), and for bringing references \cite{nem,pock_iccv,yanming_2018} to my attention. 
\end{acknowledgements}

%
 \section*{Conflict of interest}

The authors declare that they have no conflict of interest.

\bibliographystyle{spmpsci}      
\bibliography{refs}   

\end{document}